\documentclass[10pt,twocolumn]{article}

\usepackage[BFP,conBIB]{stylefile_FR}
\usepackage{tikz-cd}
\usepackage{enumerate}

\renewcommand{\thesection}{\Roman{section}} 
\renewcommand{\thesubsection}{\thesection-\Alph{subsection}}

\newcommand{\PfAlt}[1]{{\bfseries Proof of #1.}}

\DeclareMathOperator*{\argmin}{arg\,min}

\newcommand{\Norm}[2]{\norm{#2}_{#1}}
\newcommand{\pnorm}[2]{\norm{#2}_{#1}}
\newcommand{\wnorm}[2]{\norm{#2}_{\Seq{w},#1}}
\newcommand{\OPnorm}[1]{\norm{#1}_{\mathrm{op}}}
\newcommand{\pOPnorm}[2]{\norm{#2}_{\mathrm{op},#1}}

\newcommand{\CTI}{\mathrm{CTI}(\ZS^{\pm})}
\newcommand{\TI}{\mathrm{TI}(\ZS^-)}
\newcommand{\Func}{\Hc(\ZS^-)}

\title{Fading memory and the convolution theorem}
\author{Juan-Pablo Ortega\footnote{Division of Mathematical Sciences, School of Physical and Mathematical Sciences, Nanyang Technological University, Singapore} \qquad Florian Rossmannek$^*$}
\Headings{Fading memory and the convolution theorem}{Ortega and Rossmannek}

\begin{document}

\maketitle

\begin{abstract}
	Several topological and analytical notions of continuity and fading memory for causal and time-invariant filters are introduced, and the relations between them are analyzed.
A significant generalization of the convolution theorem that establishes the equivalence between the fading memory property and the availability of convolution representations of linear filters is proved.
This result extends a previous similar characterization to a complete array of weighted norms in the definition of the fading memory property.
Additionally, the main theorem shows that the availability of convolution representations can be characterized, at least when the codomain is finite-dimensional, not only by the fading memory property but also by the reunion of two purely topological notions that are called minimal continuity and minimal fading memory property.
Finally, when the input space and the codomain of a linear functional are Hilbert spaces, it is shown that minimal continuity and the minimal fading memory property guarantee the existence of interesting embeddings of the associated reproducing kernel Hilbert spaces.
\end{abstract}

\keys{approximation theory, convolution theorem, duality, fading memory property, kernel learning, linear I/O system, linear control, representer theorem, system identification.}


\section{Introduction}

Fading memory is a natural modeling assumption that has played a key role in approximation and realization results that are central in control, signal treatment, filtering, and systems theory.
Given a causal and time-invariant input/output system, the various notions of fading memory in the literature intended to mathematically encode the idea that the dependence of an output on a past input fades out as this input is further and further into the past.
Behavior of this type is expected in the description of many phenomena in the natural and social sciences and is formulated as such in classical works that go back to Volterra \cite{Volterra1959} and Wiener \cite{Wiener1958}.

In systems theory, linear time-invariant (LTI) systems constitute an important and much-used class.
Historically, it was always informally assumed that the response of such a system is always a convolution of the impulse.
This modeling assumption is crucial for traditional impulse-response (IR) analyses in filter estimation.
Even though this assumption is valid when the input dependence is defined on finite time intervals, it was pointed out for the first time in \cite{BoydChua1985} that it is actually false, in general, when those intervals are infinite.
Indeed, the availability of a convolution representation is tied to the idea that the system's filter can be estimated from a finite number of IR-samples.
However, for an LTI system defined on an infinite time interval, if the output depends on the infinite past of its input, then it is impossible to estimate its filter using only a finite number of IR-samples.
To resolve this issue, the foundational work \cite{BoydChua1985} provided the first mathematical treatment of this classical assumption, proving that the availability of a convolution representation is, in fact, linked to the idea that the system exhibits fading memory.
This statement is called the {\bfi convolution theorem}.\footnote{
This theorem should not be mistaken with the result on the Fourier transform of the convolution of two functions.}
On the one hand, convolution representations are the modeling assumption that enables a mathematical treatment of IR-analysis.
On the other hand, fading memory is a concept that speaks to us on an intuitive level and which most physical input/output systems ought to possess \cite{Sepulchre2021} when relating present outputs with inputs distant in the past.
The LTI system mentioned above, in which the output depends on the infinite past of the input, clearly violates the intuition of fading memory.
It was Boyd and Chua \cite{BoydChua1985} who gave this intuition a mathematically precise description.
This description is based on a certain analytically defined notion of continuity.
One may wonder why this specific technical definition captures exactly our intuition of the concept of fading memory.
The present work deepens our understanding of various mathematical implementations of the concept of fading memory.
In particular, we introduce a mathematical treatment of fading memory that is closer to our intuition.
A more comprehensive analysis of the relationship to convolution representations is conducted and \textit{a significant generalization of the convolution theorem constitutes the main contribution of the this work.}

Classically, systems theory has often been concerned with state-space models \cite{HoKalman1966, Gevers2006, Ljung2010}.
With the rise of machine learning, kernel methods entered the field as a new player \cite{DeNicolaoPillonetto2008, PillonettoDeNicolao2010, PillonettoQuangChiuso2011, PillonettoEtal2014, Dinuzzo2015} presenting a new method for bypassing the need for state-space representations.
One of the biggest upshots of kernel methods is the representer theorem \cite{ChristStein2008}, which transforms infinite-dimensional optimization problems into tractable finite-dimensional ones.
The theory of kernels and their associated reproducing kernel Hilbert spaces (RKHSs) was introduced to the field some time ago but has received renewed popularity in recent years \cite{LjungChenMu2020, vanWaardeSepulchre2023, HuoChaffeySepulchre2024}.
However, as noted in \cite{vanWaardeSepulchre2023}: ``the use of kernel-based system identification in system analysis and control design has so far remained limited because it faces the same challenge as machine learning: identified models can serve control design only if they come with suitable input-output properties.''
To mitigate this issue, van Waarde and Sepulchre \cite{vanWaardeSepulchre2023} studied incremental integral quadratic constraints, which are a tool to study the robustness of a system \cite{MegretskiRantzer1997, KulkarniSafonov2002}.
Like conventional analyses, many kernel methods for LTI systems are built on the premise that convolution representations are available.
By the convolution theorem, the LTI systems for which kernel methods are applicable are thus systems with fading memory, as noted in \cite{HuoChaffeySepulchre2024}.
The deepened understanding of these systems obtained in our generalized convolution theorem uncovers more of the desired input/output properties \cite{Sepulchre2021} and leads to new theoretical insights about kernel methods and the RKHSs governed by the system's fading memory properties, the second contribution of this work.

Beyond system identification, linear filters and their convolution representations have recently regained significant attention with the rise of state-space models as an alternative to the Transformer model \cite{VaswaniEtal2017} for the machine learning objective of long sequence modeling.
Various enhancements of state-space models have enabled them to overcome previous challenges and achieve competitive results in different benchmarks \cite{PatroAgneeswaran2024b}.
Among the various models that have been proposed, Mamba \cite{GuDao2023} and subsequent models based on it, such as SiMBA \cite{PatroAgneeswaran2024a}, have gained popularity due to their remarkable balance of performance and efficiency.
The core of these models is a linear state-space system that is typically assumed to admit a convolution representation to permit both a linear recurrent and a faster kernel-based implementation.
Our work supports these models by providing a theoretical background for the underlying framework.

In the context of classical dynamical systems theory, hypotheses on the memory behavior of certain systems have been central in formulating approximation results for systems defined on infinite time intervals.
This was done for the first time in the seminal paper \cite{BoydChua1985}, which extended to infinite time intervals a classical theorem \cite{Rugh1981} on the uniform approximation abilities of Volterra series for functions defined on finite time intervals.
In a similar spirit, \cite{RC6} showed that any (nonlinear!) fading memory filter can be approximated by a linear(!) system endowed with a polynomial readout.
Thus, LTI systems play a central role even in the study of general nonlinear filters, which warrants us to revisit the convolution theorem.
Similar approximation results have been obtained for recurrent neural networks assuming first finite memory \cite{Matthews1992, Matthews1993, Perryman1996, Sandberg1991IEEE, Sandberg1991MdSSP, StubberudPerryman1997} and, later on, fading memory \cite{RC7, RC12, RC20}.
Other results in connection with the universal approximation properties of recurrent neural systems that also identify fading memory as a key hypothesis are \cite{MaassJoshiSontag2007, MaassNatschMarkram2002, MaassNatschMarkram2004, MaassSontag2000}.
In a nutshell, the main mathematical feature that fading memory brings in its wake is that uniformly bounded semi-infinite input sequences form a compact space with respect to the topology induced by the weighted norms used to define fading memory despite the infinite dimensionality of the corresponding space.
This is a key fact needed in classical statements like the Stone-Weierstrass theorem, which is at the core of many universal approximation results.
Additionally, fading memory can also be found in the literature in relation to the availability and construction of temporal extensions of filters \cite{Borys2015, ColemanMizel1968, Sandberg2001, Sandberg2003} and the stability of certain differential equations \cite{ZangIglesias2004}.

The contents of this paper are divided into three sections.
Section \ref{section_fmp_notions} introduces various possible topological and analytical notions of fading memory that generalize the original definition in \cite{BoydChua1985} and spells out the relations between them.
Section \ref{section_linear_fmp_functionals} is the core of the paper and studies the interplay, in the context of linear functionals, between the availability of what are called formal and proper convolution representations and various types of fading memory introduced with different weighted norms.
In particular, Theorem \ref{wFMP_fin_dim} contains a significant generalization of the convolution theorem introduced in \cite{BoydChua1985} for linear functionals that have a fading memory property defined using weighted $\ell^p$-norms with a fixed $p \in [1, \infty]$.
We recall that the original result \cite[Thm.\ 5]{BoydChua1985} was only proved for the case $p=\infty $ and a 1-D domain.
Additionally, our theorem shows that the availability of convolution representations can be characterized, at least when the codomain is finite-dimensional, not only by the fading memory property but also by the reunion of two purely topological notions that we shall denominate as minimal continuity and minimal fading memory property.
A stronger type of fading memory is shown to be equivalent to finite memory for linear functionals.
One of the main messages of Section \ref{section_linear_fmp_functionals} is that one ought to think carefully about how to define fading memory in one's context and that even the domain of the functional under consideration can have a significant effect.
Finally, Section \ref{section_kernel_learning} considers the case in which the input space $\Zc$ and the codomain of linear functionals are defined on Hilbert spaces and studies the natural RKHSs determined by those functionals.
The minimal continuity and minimal fading memory property introduced in the previous section appear here as conditions that allow the embedding of the RKHS in $\ell^2(\Zc)^*$.
Finally, kernel learning of linear functionals is linked to direct kernel learning of convolution representations.
Proofs and technical results can be found in the appendices.


\section{Notions of fading memory}
\label{section_fmp_notions}

Several of the results in this paper will be developed for general, possibly infinite-dimensional vector spaces.
The reader is invited to think of the Euclidean cases $\R$ or $\R^d$ for clarity.
Suppose $(\Zc,\norm{\cdot})$ and $(\Yc,\Norm{\Yc}{\cdot})$ are normed vector spaces over $\R$.
Let $\Bc = \{ z \in \Zc \colon \norm{z} \leq 1 \}$ be the closed unit ball in $\Zc$.
Denote by $\N = \{1,2,\dots\}$ the set of strictly positive integers and by $\Z_- = \Z \backslash \N$ the set of non-positive integers.
Underlined symbols like $\Seq{z}$ denote sequences in $\Zc^{\Z_-} $, and $\seq{z}{t} \in \Zc$ the $t$-th entry, $t \in \mathbb{Z}_{-} $, of $\Seq{z}$. 
For any $t \in \Z_-$, let $\delta^t \colon \Zc \rightarrow \Zc^{\Z_-}$ be the {\bfi inclusion} $\delta^t(z) = (\dots,0,z,0,\dots)$ that inserts $z \in \Zc$ at the time entry $t$ into the zero vector.

\begin{assumption}
\label{asmpt_sequ_space}
	Throughout, we fix a subset $\ZS \subseteq \Zc^{\Z_-}$ with the following properties.
\begin{enumerate}[\upshape (i)]\itemsep=0em
\item
The set $\ZS$ is convex and symmetric; the latter meaning $\ZS = \{ -\Seq{z} \colon \Seq{z} \in \ZS \}$.

\item
The set $\ZS$ contains the inclusions of the unit ball $\bigcup_{t \leq 0} \delta^t(\Bc)$.

\item
For any $\Seq{z} \in \ZS$ and any subset $J \subseteq \Z_-$, the set $\ZS$ contains the sequence $\sum_{t \in J} \delta^t(\seq{z}{t})$, in which $\seq{z}{t}$ is replaced by 0 for all $t \in \Z_- \backslash J$.
\end{enumerate}
\end{assumption}

Historically, the main objects of interest have been time-invariant {\bfi filters} $\ZS \rightarrow \Yc^{\Z_-}$ with a shift-invariant domain.
We recall that such filters are in a 1:1 correspondence with {\bfi functionals} $\ZS \rightarrow \Yc$.
This equivalence justifies that we restrict our attention to functionals throughout.
Details on this equivalence can be found in Appendix \ref{appendix_funfil}.
We emphasize that the denomination `functional' is sometimes reserved for the case $\Yc = \R$.
In that sense, the maps that we call functional in this paper could also be considered generalized functionals.

We now study the implications of continuity of a functional $H \colon \ZS \rightarrow \Yc$ when the domain $\ZS$ is equipped with various interesting topologies.
With respect to any of these topologies, the inclusions $\delta^t \colon \Bc \rightarrow \ZS$ of the unit ball become continuous.
Thus, the following notion is the common denominator of the various notions of continuity of $H$ that we will encounter.

\begin{definition}
	A functional $H \colon \ZS \rightarrow \Yc$ is said to be {\bfi minimally continuous} if $H \circ \delta^t \colon \Bc \rightarrow \Yc$ is continuous for any $t \in \Z_-$.
\end{definition}

Our first notion of fading memory is designed to capture the basic property that a functional should possess to qualify for such a denomination.
This notion is designed to closely resemble our intuitive understanding of fading memory, namely ``that the effect of the distant past should fade away'' \cite{Sepulchre2021} or, from the perspective of IR-analyses, that an LTI system can be estimated from a finite number of IR-samples.

\begin{definition}
	A functional $H \colon \ZS \rightarrow \Yc$ is said to have the {\bfi minimal fading memory property (minimal FMP)} if $H(\sum_{t = T}^0 \delta^t(\seq{z}{t}))$ converges to $H(\Seq{z})$ as $T \rightarrow -\infty$ for any $\Seq{z} \in \ZS$.
\end{definition}

The sequence $\sum_{t = T}^0 \delta^t(\seq{z}{t})$ appearing in the definition of the minimal FMP is the one that is constantly zero up to time $T-1$ and agrees with $\Seq{z}$ from time $T$ onwards.
The minimal FMP is similar to the so-called {\bfi input forgetting property} (also known as the {\bfi unique steady-state property}) \cite{BoydChua1985, Jaeger2010, RC9}.
However, these properties are not the same, and neither one implies the other in general.
Compared to other natural notions of fading memory that we introduce below, the minimal FMP is arguably a weak property.
However, in the context of linear functionals, the minimal FMP will turn out to be surprisingly strong.

\begin{remark}
	It is customary to define the fading memory of a functional as being continuous with respect to a certain topology.
The minimal FMP can alternatively be defined in such a way.
Let us say a topology on $\ZS$ has property (P) if for any $\Seq{z} \in \ZS$ the sequence $\sum_{t = T}^0 \delta^t(\seq{z}{t})$ converges to $\Seq{z}$ with respect to this topology as $T \rightarrow -\infty$.
Given a functional $H \colon \ZS \rightarrow \Yc$, the following are equivalent.
(i)
$H$ has the minimal FMP.
(ii)
The initial topology on $\ZS$ induced by $H$ has property (P).
(iii)
$H$ is continuous with respect to some topology on $\ZS$ that has property (P).
\end{remark}

The next notion of fading memory has been used in \cite[Sect.\ IX]{BoydChua1985}, but only in the context of approximation and not in the context of the convolution theorem, a result at the core of this paper that is later on presented as Theorem \ref{wFMP_fin_dim}.

\begin{definition}
	A functional $H \colon \ZS \rightarrow \Yc$ is said to have the {\bfi product fading memory property (product FMP)} if it is continuous with respect to the subspace product topology on $\ZS$.
\end{definition}

The minimal and the product FMP are topological notions.
Nevertheless, the definition of fading memory has been historically introduced as continuity with respect to weighted $\ell^p$-norms \cite{BoydChua1985, RC9}, which is an analytical notion, and which we now recall.
Given $p \in [1,\infty]$ and a sequence $\Seq{w} \in (0,1]^{\Z_-}$, let $\ell^p_{\Seq{w}}(\Zc) \subseteq \Zc^{\Z_-}$ be the set of all sequences in $\Zc$ with $\wnorm{p}{\Seq{z}} = ( \sum_{t \leq 0} \seq{w}{t} \norm{\seq{z}{t}}^p )^{1/p} < \infty$ if $p < \infty$, respectively $\wnorm{\infty}{\Seq{z}} = \sup_{t \leq 0} \seq{w}{t} \norm{\seq{z}{t}} < \infty$.
This makes $(\ell^p_{\Seq{w}}(\Zc),\wnorm{p}{\cdot})$ normed vector spaces.
Taking $\Seq{w}$ to be constantly 1, we obtain the usual sequence spaces $(\ell^p(\Zc),\pnorm{p}{\cdot})$.
The closed unit ball in $\ell^p(\Zc)$ is denoted $\B^p(\Zc) = \{ \Seq{z} \in \ell^p(\Zc) \colon \pnorm{p}{\Seq{z}} \leq 1 \}$.
The set $c_0(\Zc) \subseteq \Zc^{\Z_-}$ contains all sequences vanishing at minus infinity.
We call $\Seq{w} \in (0,1]^{\Z_-}$ a {\bfi weighting sequence} if it is monotone and $\lim_{t \rightarrow -\infty} \seq{w}{t} = 0$.

\begin{definition}
	Let $p \in [1,\infty]$.
A functional $H \colon \ZS \rightarrow \Yc$ is said to be {\bfi $p$-continuous} if $\ZS \subseteq \ell^p(\Zc)$ and $H$ is continuous with respect to the topology induced by $\pnorm{p}{\cdot}$.
A functional $H \colon \ZS \rightarrow \Yc$ is said to have the {\bfi $p$-weighted fading memory property ($p$-weighted FMP)} if $\ZS \subseteq \ell^p(\Zc)$ and there exists a weighting sequence $\Seq{w}$ such that $H$ is continuous with respect to the topology induced by $\wnorm{p}{\cdot}$.
\end{definition}

The fading memory appearing in the pioneering work \cite{BoydChua1985} is the $\infty$-weighted FMP.
Many works in the literature adopt the same notion of fading memory (see \cite{Jaeger2010, JiangLiLiWang2023JML, Matthews1992, Matthews1993, RC7, RC9, RC16} to name just a few).
One of the first works to consider the $p$-weighted FMP for $p \neq \infty$ was \cite{RC9}.
We point out that most of these works assume the underlying space $\Zc$ to be compact.
In this case, the product FMP and the $\infty$-weighted FMP are equivalent \cite{RC7, RC9}.
When $\Zc$ is not compact, the distinction becomes important.
Although $\Zc$ is not compact in this work, we have flexibility in choosing the subset $\ZS$ of $\Zc^{\Z_-}$.
In particular, given a compact set $K \subseteq \Zc$, we could consider $\ZS \subseteq K^{\Z_-}$, effectively recovering the setups of other works.

\begin{remark}
\label{FMP_comparison}
	Consider a filter $\ZS \rightarrow \Yc$.
The product FMP implies both the minimal FMP and minimal continuity.
Furthermore, assuming $\ZS \subseteq \ell^p(\Zc)$ for a given $p \in [1,\infty]$, the following implications hold.
\begin{equation*}
\begin{tikzcd}[row sep = 2em, column sep = 4em]
	\text{product FMP} \arrow[shift left]{d} \\
	\arrow[shift left]{u}{\ZS \subseteq \B^{\infty}(\Zc)} \infty\text{-weighted FMP} \arrow{r} \arrow{d} \pgfmatrixnextcell \infty\text{-continuity} \arrow{d} \\
	p\text{-weighted FMP} \arrow{r} \arrow{d}{} \pgfmatrixnextcell p\text{-continuity} \arrow{d} \arrow["\ZS \subseteq c_0(\Zc)" description]{dl} \\
	\text{minimal FMP} \pgfmatrixnextcell \text{minimal cont.}
\end{tikzcd}
\end{equation*}
\end{remark}


\section{Linear fading memory functionals}
\label{section_linear_fmp_functionals}

\subsection{Main result}
\label{section_main}

We point out that the set $\ZS$ need not be a linear subspace of $\Zc^{\Z_-}$.
We call a functional $H \colon \ZS \rightarrow \Yc$ {\bfi linear} if it is the restriction of a (necessarily unique) linear functional defined on the linear span of $\ZS$.
Let $(L(\Zc,\Yc),\OPnorm{\cdot})$ be the space of continuous linear functions $\Zc \rightarrow \Yc$ with the operator norm induced by $\norm{\cdot}$ and $\Norm{\Yc}{\cdot}$.
For any $q \in [1,\infty)$ and $\Seq{\kappa} \in L(\Zc,\Yc)^{\Z_-}$, denote $\pOPnorm{q}{\Seq{\kappa}} = (\sum_{t \leq 0} \OPnorm{\seq{\kappa}{t}}^q)^{1/q}$ and $\pOPnorm{\infty}{\Seq{\kappa}} = \sup_{t \leq 0} \OPnorm{\seq{\kappa}{t}}$.

\begin{definition}
	We say that a functional $H \colon \ZS \rightarrow \Yc$ has a {\bfi formal convolution representation} if there exists a sequence $\Seq{\kappa} \in L(\Zc,\Yc)^{\Z_-}$ such that for all $\Seq{z} \in \ZS$ we have
\begin{equation}
\label{CR}
	H(\Seq{z})
	= \lim_{T \rightarrow -\infty} \sum_{t=T}^0 \seq{\kappa}{t}(\seq{z}{t}).
\end{equation}
We say that $H$ has a {\bfi convolution representation} or {\bfi proper convolution representation} if the series in \eqref{CR} converges absolutely for all $\Seq{z} \in \ZS$.
We say that $H$ has {\bfi finite memory} if $\seq{\kappa}{t}$ is zero for all but at most finitely many $t \in \Z_-$.
\end{definition}

Note that a functional with a formal convolution representation must be linear and that $\Seq{\kappa}$ must be unique.
Part (i) of \cref{lem_CR} motivates our introduction of the new notions of minimal continuity and minimal FMP, and it gives us a taste of how strong these seemingly weak notions turn out to be for linear functionals.
Part (ii) clarifies that the distinction between formal and proper convolution representations is irrelevant if the codomain $\Yc$ is finite-dimensional.

\begin{lemma}
\label{lem_CR}
	Consider a linear functional $H \colon \ZS \rightarrow \Yc$.
\begin{enumerate}[\upshape (i)]\itemsep=0em
\item
$H$ has a formal convolution representation if and only if it has the minimal FMP and is minimally continuous.

\item
If $\Yc$ is finite-dimensional, then $H$ has a formal convolution representation if and only if it has a proper convolution representation.
\end{enumerate}
\end{lemma}

Next, we present our main result -- a significant generalization of the original convolution theorem \cite[Thm.\ 5]{BoydChua1985}.
In the formulation of Theorem \ref{wFMP_fin_dim}, the minimal FMP and minimal continuity may be substituted by having a convolution representation as shown in the previous lemma.

\begin{theorem}[{\bf Convolution Theorem}]
\label{wFMP_fin_dim}
	Suppose $\Yc$ is finite-dimensional.
Let $p \in [1,\infty]$, and suppose that $c_0(\Zc) \cap \B^p(\Zc) \subseteq \ZS \subseteq  \ell^p(\Zc)$.
Consider a linear functional $H \colon \ZS \rightarrow \Yc$.
\begin{enumerate}[\upshape (i)]\itemsep=0em
\item
Suppose $p = 1$.
Then, $H$ has the 1-weighted FMP if and only if it has a convolution representation satisfying $\lim_{t \rightarrow -\infty} \OPnorm{\seq{\kappa}{t}} = 0$.
Furthermore, $H$ is 1-continuous if and only if it has a convolution representation satisfying $\pOPnorm{\infty}{\Seq{\kappa}} < \infty$ if and only if it has the minimal FMP and is minimally continuous.

\item
Suppose $p \in (1,\infty)$, and let $q \in (1,\infty)$ be the H{\"o}lder conjugate of $p$.
Then, $H$ has the $p$-weighted FMP if and only if it is $p$-continuous if and only if it has a convolution representation satisfying $\pOPnorm{q}{\Seq{\kappa}} < \infty$ if and only if it has the minimal FMP and is minimally continuous.

\item
Suppose $p = \infty$.
Then, $H$ has the $\infty$-weighted FMP if and only if if it has a convolution representation satisfying $\pOPnorm{1}{\Seq{\kappa}} < \infty$ if and only if it has the minimal FMP and is minimally continuous.
Furthermore, if $\ZS \subseteq c_0(\Zc)$, then $H$ has the $\infty$-weighted FMP if and only if it is $\infty$-continuous.
\end{enumerate}
\end{theorem}

If $c_0(\Zc) \cap \B^{\infty}(\Zc) \subseteq \ZS \subseteq \B^{\infty}(\Zc)$, then the product FMP admits the same characterization Theorem \ref{wFMP_fin_dim}.(iii) as the $\infty$-weighted FMP by Remark \ref{FMP_comparison}.
If the domain $\ZS$ is larger, then the product FMP is a strictly stronger notion than the $\infty$-weighted FMP.
In this case, the product FMP turns out to be equivalent to having finite memory.
Let $c_{00}(\Zc) \subseteq \Zc^{\Z_-}$ denote the set of all sequences with at most finitely many non-zero entries.

\begin{proposition}
\label{productFMP_fin_dim}
	Suppose $\Yc$ is finite-dimensional and $c_{00}(\Zc) \subseteq \ZS$.
Consider a linear functional $H \colon \ZS \rightarrow \Yc$.
\begin{enumerate}[\upshape (i)]\itemsep=0em
\item
The functional $H$ has the product FMP if and only if it has finite memory.

\item
If $\ZS = \Zc^{\Z_-}$, then $H$ has the product FMP if and only if it has finite memory if and only if it has the minimal FMP and is minimally continuous.
\end{enumerate}
\end{proposition}


\subsection{Discussion and examples}

\subsubsection{Infinite-dimensional codomain}
\label{section_inf_dim}

The finite-dimensionality of $\Yc$ in Theorem \ref{wFMP_fin_dim} and Proposition \ref{productFMP_fin_dim} is necessary.
Indeed, for any $p \in [1,\infty]$, if $\Yc = \ell^p(\Zc)$ is equipped with the topology induced by $\pnorm{p}{\cdot}$, then the identity map $\ell^p(\Zc) \rightarrow \ell^p(\Zc)$ is a $p$-continuous linear functional without the $p$-weighted FMP.
If $\Yc = \ell^p(\Zc)$ is equipped with the subspace product topology, then the identity map has the product FMP and infinite memory.
Nonetheless, some parts of Theorem \ref{wFMP_fin_dim} and Proposition \ref{productFMP_fin_dim} continue to hold if the codomain $\Yc$ is infinite-dimensional.
We know from Lemma \ref{lem_CR} that the minimal FMP and minimal continuity are equivalent to having a formal convolution representation, albeit not necessarily a proper one.
If $p < \infty$, then these are also equivalent to $p$-continuity.
The $p$-weighted FMP becomes a strictly stronger notion as seen from the identity example $\ell^p(\Zc) \rightarrow \ell^p(\Zc)$ above, and having a convolution representation satisfying $\pOPnorm{q}{\Seq{\kappa}} < \infty$ for the H{\"o}lder conjugate $q$ is the strongest notion (except if $p=1$ in which case the 1-weighted FMP is still equivalent to $\lim_{t \rightarrow -\infty} \OPnorm{\seq{\kappa}{t}} = 0$).
That a formal convolution representation may not be a proper one is linked to the fact that the duality between $\ell^p$ and $\ell^q$ spaces for H{\"o}lder conjugates $p$ and $q$ does not generalize to the infinite-dimensional case.
Details are discussed in Appendix \ref{section_app_main}.

\subsubsection{Edge cases}

If $p \in (1,\infty)$ or $p = \infty$ and $\ZS \subseteq c_0(\Zc)$, then $p$-continuity and the $p$-weighted FMP turn out to be equivalent for linear functionals.
The cases $p=1$ and $p = \infty$ with $\ZS \not\subseteq c_0(\Zc)$ are special.
We see from Theorem \ref{wFMP_fin_dim}.(i) that any convolution representation with $\pOPnorm{\infty}{\Seq{\kappa}} < \infty$ and $\limsup_{t \rightarrow -\infty} \OPnorm{\seq{\kappa}{t}} > 0$ provides an example of a linear functional that is 1-continuous but does not have the 1-weighted FMP.
For $p=\infty$, consider a Banach limit $H \colon \ell^{\infty}(\R) \rightarrow \R$, that is, $H$ is a positive, shift-invariant, $\infty$-continuous, linear functional such that $H(\Seq{z}) = \lim_{t \rightarrow -\infty} \seq{z}{t}$ whenever the limit exists \cite{Conway2007}.
This functional had been brought forth in \cite{BoydChua1985} to show, for the first time, that there are $\infty$-continuous functionals that do not have a convolution representation, nor the $\infty$-weighted FMP.
By Theorem \ref{wFMP_fin_dim}, $H$ does not even have the minimal FMP.
To see this directly, simply observe that $H(\sum_{t = T}^0 \delta^t(\seq{z}{t})) \equiv 0$, so $H$ would have to be constantly zero.
Lastly, Theorem \ref{wFMP_fin_dim} also exposes the existence of linear functionals with the $p$-weighted FMP that are not $q$-continuous for any $1 \leq p < q \leq \infty$.

\subsubsection{Curse of memory}

Consider the linear functional $H^{\omega} \colon \ell^{\infty}(\R) \rightarrow \R$ with the convolution representation $\seq{\kappa}{t}^{\omega} = (1-t)^{1+\omega}$, where $\omega \in (0,\infty)$.
Then, $\pOPnorm{1}{\Seq{\kappa}^{\omega}} < \infty$ and, hence, $H^{\omega}$ has the $\infty$-weighted FMP.
This example has been studied to illustrate the so-called curse of memory \cite{LiHanELi2020, LiHanELi2022}.
The weighting sequence $\Seq{w}$ for which $H^{\omega}$ is continuous with respect to $\wnorm{\infty}{\cdot}$ captures the rate of decay of the memory of $H^{\omega}$.
It has to satisfy $\sum_{t \leq 0} \seq{w}{t}^{-1} \abs{\seq{\kappa}{t}^{\omega}} < \infty$, which excludes any weighting sequence with exponential decay.
This is in line with the findings in \cite{LiHanELi2020, LiHanELi2022}.

\subsubsection{Linear state equations}
\label{section_linear_sss}

Let $(\Xc,\Norm{\Xc}{\cdot})$ be a Banach space, and let $A \in L(\Xc,\Xc)$ and $B \in L(\R^d,\Xc)$.
Suppose the linear state equation $\seq{x}{t} = A(\seq{x}{t-1}) + B(\seq{z}{t})$, where $t \in \Z_-$, admits a unique solution for every input sequence $\Seq{z} \in \ell^{\infty}(\R^d)$.
The associated functional $H_{A,B} \colon \ell^{\infty}(\R^d) \rightarrow \Xc$ is minimally continuous since $H_{A,B} \circ \delta^t = A^{-t} \circ B$.
If the spectral radius $\rho(A)$ is strictly smaller than 1, then the convolution representation $\seq{\kappa}{t} = A^{-t} \circ B$ satisfies $\pOPnorm{1}{\Seq{\kappa}} < \infty$, in which case $H_{A,B}$ has the $\infty$-weighted FMP.\footnote{
As discussed in Section \ref{section_inf_dim}, this part of Theorem \ref{wFMP_fin_dim} continues to hold for an infinite-dimensional codomain, see Lemma \ref{CR_to_pcont} in Appendix \ref{section_app_main}.}
Indeed, $\rho(A) = \lim_{k \rightarrow \infty} \OPnorm{A^k}^{1/k}$ by Gelfand's formula \cite{Conway2007} and, hence, given any $r \in (\rho(A),1)$ there exists some $t_0 \in \Z_-$ such that $\OPnorm{A^{-t}} \leq r^{-t}$ for all $t \leq t_0$, which implies that $\pOPnorm{1}{\Seq{\kappa}} < \infty$.
If $\Xc \cong \R^n$ is finite-dimensional, then the condition that $\rho(A) < 1$ is implicit in the existence and uniqueness of solutions, and so is the $\infty$-weighted FMP.
Indeed, if we had $\rho(A) \geq 1$, then there existed an eigenvector $v \in \C^n$ of $A$ for some eigenvalue $\lambda \in \C$ with $\abs{\lambda} \geq 1$.
In this case, both the constant zero sequence and the sequence $\seq{x}{t} = \mathrm{Re}(\lambda^t v)$ belonged to $\ell^{\infty}(\R^n)$ and posed a solution to the state equation when the input is the constant zero sequence \cite{RC16}.

The linear state equation can be extended to a state-space system with a readout $\seq{y}{t} = h(\seq{x}{t})$, where $h \in L(\Xc,\R)$.
The associated functional becomes $H_{A,B}^h = h \circ H_{A,B} \colon \ell^{\infty}(\R^d) \rightarrow \R$, which inherits the $\infty$-weighted FMP from $H_{A,B}$ if $A$ has spectral radius strictly smaller than 1.
If $A$ is nilpotent, then $H_{A,B}^h$ has finite memory and, hence, the product FMP.
Conversely, as soon as $H_{A,B}^h$ has even the minimal FMP, we must have $\sum_{t \leq 0} \OPnorm{h \circ A^{-t} \circ B} < \infty$.

\begin{remark}
	Any linear functional $H \colon \ell^{\infty}(\Zc) \rightarrow \Yc$ with the $\infty$-weighted FMP can be realized by a linear state-space system \cite{RC16}.
If $H$ has the product FMP and, hence, finite memory, then the state operator of the system can be taken nilpotent.
If, in addition, $\Zc$ and $\Yc$ are finite-dimensional, then the state-space system can be constructed on a finite-dimensional state space.
\end{remark}

Building on linear state equations, one can construct recurrent generalized Barron functionals as follows \cite{RC24}.
Let $p,q \in [1,\infty]$ be H{\"o}lder conjugates.
Take the state space to be $\Xc = \ell^q(\R)$, and assume that the solutions to the state equation belong to $\ell^{\infty}(\ell^q(\R))$.
Let $T \colon \ell^{\infty}(\R^d) \rightarrow \ell^{\infty}(\R^d)$ denote the shift operator, and let $\sigma \colon \R \rightarrow \R$ be Lipschitz continuous.
Let $\mu$ be a Borel probability measure on $\Sigma = \R \times \ell^p(\R) \times \R^d \times \R$ satisfying $\int_{\Sigma} \abs{w} ( \pnorm{p}{a} + \Norm{\R^d}{c} + \abs{b} ) \, d\mu(w,a,c,b) < \infty$.
A recurrent generalized Barron functional $H \colon \ell^{\infty}(\R^d) \rightarrow \R$ is given by
\begin{equation*}
	H(\Seq{z})
	= \int_{\Sigma} w \sigma( a \circ H_{A,B} \circ T(\Seq{z}) + c \cdot \seq{z}{0} + b ) \, d\mu(w,a,c,b),
\end{equation*}
where $a \in \ell^p(R)$ is identified with the induced dual element in $\ell^q(\R)^*$.
This functional is nonlinear in general, but it builds on the linear functional $H_{A,B}$.
Thus, if $\sum_{t \leq 0} \OPnorm{A^{-t} \circ B} < \infty$, then $H$ has the $\infty$-weighted FMP.


\section{Kernel learning with linear functionals}
\label{section_kernel_learning}

Suppose that the norms on $\Yc$ and $\Zc$ are induced by some inner products $\ip{\cdot}{\cdot}_{\Yc}$ and $\ip{\cdot}{\cdot}_{\Zc}$ that turn $\Yc$ and $\Zc$ into Hilbert spaces (possibly infinite-dimensional).
Any (not necessarily linear) functional $H \colon \ZS \rightarrow \Yc$ can be used as a feature map to induce a symmetric, positive semi-definite kernel $K_H \colon \ZS \times \ZS \rightarrow \R$ on $\ZS$ by $K_H(\Seq{z}^1,\Seq{z}^2) = \ip{H(\Seq{z}^1)}{H(\Seq{z}^2)}_{\Yc}$.
Any such kernel has an associated RKHS, which can be characterized as follows.

\begin{proposition}[Theorem 4.21 in \cite{ChristStein2008}]
\label{RKHS}
	Consider a functional $H \colon \ZS \rightarrow \Yc$.
Let $\Yc_0$ be the closure of $\mathrm{span}(H(\ZS))$ in $\Yc$, and let $(\Hb,\ip{\cdot}{\cdot}_{\Hb})$ be the RKHS associated to the kernel $K_H$.
Then, the map $y \mapsto \ip{y}{H(\cdot)}_{\Yc}$ is an isometric isomorphism of Hilbert spaces $(\Yc_0 , \ip{\cdot}{\cdot}_{\Yc} ) \rightarrow (\Hb,\ip{\cdot}{\cdot}_{\Hb})$.
\end{proposition}

Here, we show that the RKHS is isometrically isomorphic to $\ell^2(\Zc)^*$ for orthogonal linear functionals with the minimal FMP.
Recall that $\ell^2(\Zc)^*$ is a Hilbert space with an inner product $\ip{\cdot}{\cdot}_{\ell^2(\Zc)^*}$ induced by the inner product $\ip{\Seq{z}^1}{\Seq{z}^2}_{\ell^2(\Zc)} = \sum_{t \leq 0} \ip{\seq{z}{t}^1}{\seq{z}{t}^2}_{\Zc}$ on $\ell^2(\Zc)$ and the Riesz representation theorem.

\begin{proposition}
\label{RKHS_l2}
	Suppose $\B^2(\Zc) \subseteq \ZS$.
Consider a linear functional $H \colon \ZS \rightarrow \Yc$.
Let $\Yc_0$ be the closure of $\mathrm{span}(H(\ZS))$ in $\Yc$, and let $\Hb$ be the RKHS associated to the kernel $K_H$.
Suppose $H$ has the minimal FMP and is minimally continuous.
Then, $\Hb$ embeds linearly in $\ell^2(\Zc)^*$ by mapping a function $f \in \Hb$ to the unique linear extension of $f|_{\B^2(\Zc)}$, and $\mathrm{span}(H(\B^2(\Zc)))$ is dense in $\Yc_0$.
Furthermore, if the restriction $H \colon (\B^2(\Zc),\ip{\cdot}{\cdot}_{\ell^2(\Zc)}) \rightarrow (\Yc,\ip{\cdot}{\cdot}_{\Yc})$ is orthogonal, then the embedding is an isometric isomorphism of Hilbert spaces $(\Hb,\ip{\cdot}{\cdot}_{\Hb}) \rightarrow (\ell^2(\Zc)^*,\ip{\cdot}{\cdot}_{\ell^2(\Zc)^*})$, and $\mathrm{span}(H(\B^2(\Zc))) = \Yc_0$.
\end{proposition}

Suppose we are given finitely many pairs of data/target points $(\Seq{z}^i,\omega_i) \in \ZS \times \R$, $i = 1,\dots,M$, a strictly increasing function $\Lambda \colon [0,\infty) \rightarrow [0,\infty)$, playing the role of a regularizer, and a function $\Ec \colon \R^{2M} \rightarrow \R$, which is typically an empirical risk.
Consider the (regularized empirical risk) minimization problem
\begin{equation}
\label{optim_kernel}
	\inf_{f \in \Hb} \, \Ec( f(\Seq{z}^1),\omega_1,\dots,f(\Seq{z}^M),\omega_M ) + \Lambda(\Norm{\Hb}{f}).
\end{equation}
The representer theorem \cite{ChristStein2008} states that any minimizer $\hat{f}$, if existent, takes the form $\hat{f} = \sum_{i=1}^M \hat{\alpha}_i \ip{H(\Seq{z}^i)}{H(\cdot)}_{\Yc}$ for some $\hat{\alpha}_1,\dots,\hat{\alpha}_M \in \R$.
This is a well-known method to reduce an a priori infinite-dimensional minimization problem to a finite-dimensional one \cite{SchoelHerbSmola2001}.
The coefficients $\hat{\alpha}_i$ can be estimated as the solution of a regression problem that admits a closed form solution in terms of the evaluations $\ip{H(\Seq{z}^i)}{H(\Seq{z}^j)}_{\Yc}$ if $\Ec$ is a mean square error and $\Lambda$ a Tikhonov regularizer \cite{RC25, Wahba1990}.
Let us discuss an explicit example of this.

\subsection*{Kernels in linear system identification}

In linear system identification, one faces the empirical risk minimization problem
\begin{equation}
\label{ERM_CR}
	\hat{\Seq{\kappa}}
	= \argmin_{\Seq{\kappa} \in \Hc} \, \sum_{i=1}^M \left( \omega_i - \sum_{t \leq 0} \seq{\kappa}{t} \cdot \seq{z}{t}^i \right)^2 + \mathrm{regularizer}
\end{equation}
with (possibly noisy) 1-D observations $\omega_1,\dots,\omega_M \in \R$ and input samples $\Seq{z}^1,\dots,\Seq{z}^M \in \ell^{\infty}(\R^d)$.
Here, $\Hc$ is some hypothesis class $\Hc \subseteq (\R^d)^{\Z_-}$.
That the task is formulated as minimizing over the coefficients of a convolution representation is telling of the fact that the availability of convolution representations has been a classical assumption.
A popular choice for the hypothesis class are subsets of sequence space that come from the RKHS of a chosen kernel, in which case the RKHS norm also lends itself as a suitable regularizer.
Consider, for example, the use of a Gaussian kernel, likely the most used model for kernel-based regularization.
As noted in \cite{PillonettoQuangChiuso2011}, the RKHS of the Gaussian kernel is not a subset of $\ell^1(\R^d)$.
An undesirable consequence is that choosing the RKHS of the Gaussian kernel as the hypothesis space does not preserve the bounded-input bounded-output (BIBO) stability.
This motivated \cite{DeNicolaoPillonetto2008, PillonettoDeNicolao2010} to introduce stable spline kernels and show that their RKHS satisfies BIBO stability.
In \cite{DeNicolaoPillonetto2008, PillonettoDeNicolao2010} and in the subsequent works \cite{PillonettoQuangChiuso2011, PillonettoEtal2014, Dinuzzo2015}, it is the coefficients of the convolution representation that are learned by means of kernelization.
In particular, the kernels under consideration are functions on $\R^d \times \R^d$.
Below, we discuss the alternative in which we apply the kernel trick not to the coefficients of the convolution representation but to the functional itself.
In this case, the kernel will be a function on $\ell^2(\R^d) \times \ell^2(\R^d)$.
However, we will see that the efficient computability of the kernel is not compromised by the infinite-dimensionality of its domain.

\subsection*{The recursive linear kernel}

In the notation of the previous sections, suppose $\Zc = \R^d$ and $\Xc = \Yc = \ell^2(\R^d)$ with the standard inner products.
Let $T^{-1} \colon \Xc \rightarrow \Xc$ denote the right-inverse $T^{-1}(\Seq{z}) = (\dots,\seq{z}{-1},\seq{z}{0},0)$ of the shift operator $T$.
Given $\lambda \in (0,1)$, consider the linear state equation $\seq{x}{t} = \lambda T^{-1}(\seq{x}{t-1}) + \delta^0(\seq{z}{t})$.
The associated functional $H^{\lambda} \colon \ell^{\infty}(\R^d) \rightarrow \ell^2(\R^d)$ is given by $H^{\lambda}(\Seq{z}) = \sum_{t \leq 0} \lambda^{-t} \delta^t(\seq{z}{t})$, see Section \ref{section_linear_sss}.
The induced kernel is $K_{\lambda}(\Seq{z}^1,\Seq{z}^2) = \sum_{t \leq 0} \lambda^{-2t} \ip{\seq{z}{t}^1}{\seq{z}{t}^2}_{\R^d}$, and it satisfies
\begin{equation}
\label{kernel_rec}
	K_{\lambda}(\Seq{z}^1,\Seq{z}^2)
	= \ip{\seq{z}{0}^1}{\seq{z}{0}^2}_{\R^d} + \lambda^2 K_{\lambda}(T(\Seq{z}^1),T(\Seq{z}^2)).
\end{equation}
Let $(\Hb^{\lambda},\ip{\cdot}{\cdot}_{\Hb^{\lambda}})$ be the RKHS associated to the kernel $K_{\lambda}$.
Since the image of $H^{\lambda}$ is dense in $\ell^2(\R^d)$, Proposition \ref{RKHS} conveys that $\ell^2(\R^d)$ is isometrically isomorphic to $\Hb^{\lambda}$ via $\Seq{y} \mapsto \ip{\Seq{y}}{H^{\lambda}(\cdot)}_{\ell^2(\R^d)}$.
The functional $H^{\lambda}$ has the minimal FMP and is minimally continuous.
The linear embedding of $\Hb^{\lambda}$ into $\ell^2(\R^d)^*$ promised by Proposition \ref{RKHS_l2} is given by $\ip{\Seq{y}}{H^{\lambda}(\cdot)}_{\ell^2(\R^d)} \mapsto \ip{H^{\lambda}(\Seq{y})}{\cdot}_{\ell^2(\R^d)}$.
However, this is not an isometry.
Indeed, $H^{\lambda}$ is not orthogonal as a map.
If we take $\ell^2(\R^d)$ as the domain of $H^{\lambda}$ instead of $\ell^{\infty}(\R^d)$, then we can permit $\lambda = 1$.
In this case, $H^1$ is the identity on $\ell^2(\R^d)$, which is trivially orthogonal.
The isometric isomorphism $\ell^2(\R)^d \rightarrow \Hb^1 \cong \ell^2(\R^d)^*$ from Proposition \ref{RKHS} becomes an instance of the Riesz representation theorem.
The linear embedding $\Hb^1 \hookrightarrow \ell^2(\R^d)^*$ from Proposition \ref{RKHS_l2} is now simply the identity.
Now, consider the kernelized minimization problem \eqref{optim_kernel} with $\Ec$ the empirical risk from \eqref{ERM_CR}, that is,
\begin{equation}
\label{ERM_kernel}
	\inf_{f \in \Hb^{\lambda}} \, \sum_{i=1}^M ( \omega_i - f(\Seq{z}^i) )^2 + \Lambda(\Norm{\Hb^{\lambda}}{f}).
\end{equation}
We have noted above that $\Hb^{\lambda}$ embeds linearly in $\ell^2(\R^d)^*$.
Thus, taking $\Hb^{\lambda}$ as our hypothesis class ensures BIBO stability.
By Proposition \ref{RKHS}, the minimization problem \eqref{ERM_kernel} is equivalent to
\begin{equation*}
	\inf_{y \in \ell^2(\R^d)} \, \sum_{i=1}^M ( \omega_i - \ip{y}{H^{\lambda}(\Seq{z}^i)}_{\ell^2(\R^d)} )^2 + \Lambda(\Norm{\ell^2(\R^d)}{y}).
\end{equation*}
This shows that \eqref{ERM_CR} with hypothesis class $\Hc = \ell^2(\R^d)$ corresponds to the special case $\lambda = 1$ in \eqref{ERM_kernel} (though, the regularizer may differ).
We formulated the minimization problems with given samples $\Seq{z}^1,\dots,\Seq{z}^M \in \ell^{\infty}(\R^d)$.
In practice, one does not have access to samples $\Seq{z}^i$ with an infinite past $(\dots,\seq{z}{-2}^i,\seq{z}{-1}^i,\seq{z}{0}^i)$.
Rather, one is given the final $M$ instances $\seq{z}{-M+1},\dots,\seq{z}{0} \in \R^d$ of an input trajectory $\Seq{z} \in \ell^{\infty}(\R^d)$, and the samples $\Seq{z}^1,\dots,\Seq{z}^M$ used in the minimization problem are $\Seq{z}^i = (\dots,0,0,\seq{z}{-M+1},\dots,\seq{z}{-M+i})$, which is motivated by a state-space system representation \cite{RC16}.
By the representer theorem, we are looking for a minimizer of the form $\hat{f} = \sum_{i=1}^M \alpha_i \ip{H^{\lambda}(\Seq{z}^i)}{H^{\lambda}(\cdot)}_{\ell^2(\R^d)}$ for some $\alpha_1,\dots,\alpha_M \in \R$.
For a Tikhonov regularizer $\Lambda(x) = \gamma x$, $\gamma > 0$, the minimization problem then admits the vectorized closed-form solution \cite{Wahba1990}
\begin{equation*}
	(\hat{\alpha}_1,\dots,\hat{\alpha}_M)^T
	= ( G_{\lambda} + \gamma I )^{-1} (\omega_1,\dots,\omega_M)^T,
\end{equation*}
where $I \in \R^{M \times M}$ is the identity matrix and $G_{\lambda} \in \R^{M \times M}$ is the Gram matrix with entries $K_{\lambda}(\Seq{z}^i,\Seq{z}^j)$, $i,j = 1,\dots,M$.
The strength of the linear kernel $K_{\lambda}$ is the recursive formula \eqref{kernel_rec}.
This recursive structure of the kernel makes the Gram matrix efficiently computable even if the sample size $M$ is large -- a property that only few induced kernels are known to possess \cite{RC25}.
Moreover, the minimization problem \eqref{ERM_kernel} comes with an additional hyperparameter $\lambda$, which generalizes \eqref{ERM_CR}.
If $\lambda < 1$, then one can see from \eqref{kernel_rec} that the accuracy of the recursive computation improves exponentially in the number of recursions.
Thus, the Gram matrix can be computed to a high accuracy even without using the full history of given samples, which is useful in cases where $M$ is very large.


\section{Conclusion}
\label{Conclusion}

In this paper, we have conducted a detailed analysis of various types of continuity of linear functionals.
The main result linked various so-defined fading memory properties to the availability of convolution representations.
This greatly generalizes the classical convolution theorem that first appeared in \cite{BoydChua1985}.
The most striking features are, on the one hand, the central role of the two topological notions of minimal continuity and the minimal FMP, especially since both of these notions do not depend on the topology of the domain of the functional.
This shows that the concept of fading memory is closely tied to the domain of the functional as a set, for example, having the minimal FMP on $\ell^{\infty}(\Zc)$ has greater implications than having the minimal FMP on $\ell^1(\Zc)$.
On the other hand, continuity with respect to the product topology turns out to be equivalent to finite memory if the sequences in the input space are not uniformly bounded.
This highlights the importance of giving careful consideration to which notion of fading memory to use in one's context, and it expands our understanding of the common implicit modeling assumption in system identification of linear systems.
The strength of the two topological notions of minimal continuity and minimal fading memory have been further leveraged to embed RKHSs associated to such linear functionals defined on Hilbert spaces.
Kernel learning of linear functionals has been linked to previous kernel approaches, in which the coefficients of convolution representations have been kernelized directly.

An immediate question is whether an analogous convolution theorem holds in a continuous-time setting, in which the inputs are not sequences indexed by integers but functions of the (half-axis of the) real line.
Parts of our convolution theorem are based on the duality of the $\ell^p$-sequence spaces.
We expect these results to transfer easily to the continuous-time setting since analogous duality results hold for the $L^p$-function spaces.
However, the continuous-time and the discrete-time case are not conceptually equivalent in all aspects.
For example, the product FMP exhibits fundamentally different behaviour in continuous-time.
We have seen in Proposition \ref{productFMP_fin_dim} that the product FMP is equivalent to finite memory in discrete-time.
In continuous-time, the product FMP and the availability of a convolution representation are incompatible.
A proof of this can be found in Appendix \ref{appendix_cont}.
Morally, the key difference is that any measure on the integers is absolutely continuous with respect to the counting measure.
On the real line, on the other hand, functionals that are dual to measures that are not absolutely continuous with respect to the Lebesgue measure cannot admit a convolution representation but can still have fading memory.
Thus, although some concepts from the discrete-time case carry over, a careful study of the continuous-time case is desirable.

The idea of having different notions of fading memory as well as using different weighting sequences in the concept of the $p$-weighted FMP lends itself to the study of the so-called curse of memory.
This curse of memory has been approached in continuous-time in \cite{LiHanELi2020, LiHanELi2022}.
However, as we pointed out in the preceding paragraph, one ought to be careful with applying intuition from discrete-time functionals to their continuous-time counterparts.
One method of attack to the curse of memory are lower and upper bounds on approximation rates.
Another method of attack would be to study continuity of a functional with respect to a weighted $\ell^p$-norm and at the same time discontinuity with respect to the $\ell^p$-norm weighted by a different weighting sequence.

In linear system identification, another key result has been Willems' fundamental lemma \cite{WillemsEtal2005}.
This result has often been applied in a state-space system context \cite{BerbAllg2020, DePersisTesi2020, vanWaardeEtal2020}.
To extend the applicability of the lemma, it is of interest to identify controllability conditions for fading memory systems that admit solely a convolution representation but no state-space system representation.
Willems' fundamental lemma has also recently been studied in a kernel framework \cite{MoloFaul2024}, which needs to be investigated further.

Finally, it is of great interest whether the ideas from this paper find applications in the study of nonlinear functionals.
For nonlinear functionals, the analog of a convolution representation could be a Volterra series representation.
Such representations have been proved to exist in \cite{RC9,RC25} and also found their way to system identification \cite{DallaLCarliPillonetto2021}.
We hypothesize that some of the assumptions in those results can instead be replaced by careful fading memory assumptions.

\acks{The authors acknowledge partial financial support from the School of Physical and Mathematical Sciences of the Nanyang Technological University.
The second author is funded by an Eric and Wendy Schmidt AI in Science Postdoctoral Fellowship at the Nanyang Technological University.
We thank Lyudmila Grigoryeva and Qianxiao Li for helpful discussions and remarks.}


\appendix

\section{Technical details for Section III}
\label{section_app_main}

\subsection*{The general case}

This section covers the results in which the codomain $\Yc$ may be infinite-dimensional.
We begin by showing that the minimal FMP and minimal continuity are equivalent to having a formal convolution representation.

\begin{proof}[\PfAlt{Lemma \ref{lem_CR}.(i)}]
	Suppose $H$ has the minimal FMP and is minimally continuous.
Let $\tilde{H} \colon \mathrm{span}(\ZS) \rightarrow \Yc$ be the unique linear extension of $H$, and define $\seq{\kappa}{t} = \tilde{H} \circ \delta^t$.
Restricted to the unit ball, $\seq{\kappa}{t} = H \circ \delta^t$ is continuous.
Being a linear map, continuity on the unit ball is equivalent to continuity on all of $\Zc$.
Also by linearity, $H ( \sum_{t = T}^0 \delta^t(\seq{z}{t}) ) = \sum_{t = T}^0 \seq{\kappa}{t}(\seq{z}{t})$ for any $\Seq{z} \in \ZS$ and $T \in \Z_-$.
Taking the limit and applying the minimal FMP, we obtain a formal convolution representation of $H$.
The converse implication is clear.
\end{proof}


\begin{lemma}
\label{CR_to_pcont}
	Let $p \in [1,\infty]$, and suppose $\B^p(\Zc) \subseteq \ZS \subseteq \ell^p(\Zc)$ or $c_0(\Zc) \cap \B^p(\Zc) \subseteq \ZS \subseteq c_0(\Zc) \cap \ell^p(\Zc)$.
Consider a linear functional $H \colon \ZS \rightarrow \Yc$ with a formal convolution representation $\Seq{\kappa}$.
\begin{enumerate}[\upshape (i)]\itemsep=0em
\item
The functional $H$ is $p$-continuous.

\item
If $p \neq 1$ and $\pOPnorm{q}{\Seq{\kappa}} < \infty$ for the H{\"o}lder conjugate $q \in [1,\infty)$ of $p$, then $H$ has the $p$-weighted FMP.
\end{enumerate}
\end{lemma}

\begin{proof}
	(i)
Note that $\mathrm{span}(\ZS) = \ell^p(\Zc)$ or $\mathrm{span}(\ZS) = c_0(\Zc) \cap \ell^p(\Zc)$, which is a Banach space with norm $\pnorm{p}{\cdot}$ in either case.
For each $n \in \N$, the finite memory linear functional $H_n \colon \mathrm{span}(\ZS) \rightarrow \Yc$ given by $H_n(\Seq{z}) = \sum_{t = -n}^0 \seq{\kappa}{t}(\seq{z}{t})$ is $p$-continuous.
Let $\tilde{H}$ be the unique linear extension of $H$ to $\mathrm{span}(\ZS)$.
Since $H_n$ converges point-wise to $\tilde{H}$, the sequence $(\Norm{\Yc}{H_n(\Seq{z})})_n$ is bounded for all $\Seq{z} \in \mathrm{span}(\ZS)$.
By the principle of uniform boundedness, the operator norm of $H_n$ with respect to $\pnorm{p}{\cdot}$ is uniformly bounded in $n$.
This implies that the point-wise limit of $H_n$ has finite operator norm with respect to $\pnorm{p}{\cdot}$, that is, $\tilde{H}$ is $p$-continuous.
In particular, $H$ is $p$-continuous as a restriction of $\tilde{H}$.

(ii)
Let $r = 1/p$ if $p < \infty$ and $r = 1$ if $p = \infty$.
Take a weighting sequence $\Seq{w}$ such that $C := \sum_{t \leq 0} \seq{w}{t}^{-qr} \OPnorm{\seq{\kappa}{t}}^q < \infty$.\footnote{
The existence of such a weighting sequence is a fact from calculus:
given a convergent series $\sum_{t \leq 0} a_t$ with $(a_t)_t \subseteq [0,\infty)$, there exists an unbounded monotone sequence $(b_t)_t \subseteq (1,\infty)$ such that $\sum_{t \leq 0} a_t b_t < \infty$.
In our case, $a_t = \OPnorm{\seq{\kappa}{t}}^q$ and $\seq{w}{t} = b_{t}^{-1/qr}$.
}
Then, H{\"o}lder's inequality shows that $H$ has the $p$-weighted FMP;
\begin{equation*}
	\Norm{\Yc}{ H(\Seq{z}) }
	\leq \sum_{t \leq 0} \seq{w}{t}^{-r} \OPnorm{\seq{\kappa}{t}} \seq{w}{t}^r \norm{\seq{z}{t}}
	\leq C^{1/q} \wnorm{p}{\Seq{z}}.
\end{equation*}
\end{proof}


\begin{proposition}
\label{1wFMP}
	Suppose $\ZS \subseteq \ell^1(\Zc)$.
Then, a linear functional $H \colon \ZS \rightarrow \Yc$ has the 1-weighted FMP if and only if it has a convolution representation satisfying $\lim_{t \rightarrow -\infty} \OPnorm{\seq{\kappa}{t}} = 0$.
\end{proposition}

\begin{proof}
	Suppose $H$ has a convolution representation satisfying $\lim_{t \rightarrow -\infty} \OPnorm{\seq{\kappa}{t}} = 0$.
Consider the monotone sequence $\seq{w}{t} = \sup_{s \leq t} \min\{ 1 , \OPnorm{ \seq{\kappa}{s} } \}$.
If $\Seq{w}$ is eventually zero, then $H$ has finite memory, which clearly implies the 1-weighted FMP.
If $\Seq{w}$ is not eventually zero, then it is a weighting sequence.
With this weighting sequence, $\Norm{\Yc}{ H(\Seq{z}) } \leq C \wnorm{1}{\Seq{z}}$, where $C = \sup_{t \leq 0} \max\{ 1 , \OPnorm{\seq{\kappa}{t}} \}$.
This shows that $H$ has the 1-weighted FMP.
Conversely, if $H$ has the 1-weighted FMP, then there exists a weighting sequence $\Seq{w}$ such that $H$ is continuous with respect to $\wnorm{1}{\cdot}$.
Take $(\eta_n)_{n \in \N} \subseteq (0,1)$ such that $\Norm{\Yc}{H(\Seq{z})} \leq n^{-1}$ for all $\Seq{z} \in \ZS$ with $\wnorm{1}{\Seq{z}} \leq \eta_n$.
We know from Remark \ref{FMP_comparison} and Lemma \ref{lem_CR}.(i) that $H$ has a formal convolution representation with some $\Seq{\kappa} \in L(\Zc,\Yc)^{\Z_-}$.
Let $\epsilon > 0$ and $n \geq \epsilon^{-1}$.
Take $T \in \Z_-$ with $\seq{w}{T} \leq \eta_n$.
For any $t \leq T$ and any $z \in \Zc$ with $\norm{z} = 1$, we have $\wnorm{1}{\delta^t(z)} = \seq{w}{t} \leq \eta_n$ and, hence, $\Norm{\Yc}{ \seq{\kappa}{t}(z) } = \Norm{\Yc}{ H( \delta^t(z) ) } \leq n^{-1} \leq \epsilon$.
Thus, $\OPnorm{\seq{\kappa}{t}} \rightarrow 0$ as $t \rightarrow -\infty$.
In particular, $\pOPnorm{\infty}{\Seq{\kappa}} < \infty$, which implies that the formal convolution representation is a proper one.
\end{proof}

\begin{example}
	Let $(\Vc,\Norm{\Vc}{\cdot})$ and $(\Wc,\Norm{\Wc}{\cdot})$ be two normed vector spaces over $\R$.
Take $\Zc = L(\Vc,\Wc)$ to be the vector space of continuous linear maps from $\Vc$ to $\Wc$ with the operator norm, and $\Yc = \ell^1(\Wc)$ with the $\ell^1$-norm.
Fix a bounded sequence $\Seq{v} \in \ell^{\infty}(\Vc)$, and consider the sequential evaluation functional
\begin{equation*}
	H_{\Seq{v}} \colon \ell^1(L(\Vc,\Wc)) \rightarrow \ell^1(\Wc),
	\quad \Seq{A} \mapsto (\seq{A}{t}(\seq{v}{t}))_{t \leq 0}.
\end{equation*}
Then, $H_{\Seq{v}}$ is minimally continuous, has the minimal FMP, and has the formal convolution representation $\seq{\kappa}{t}(A) = H \circ \delta^t(A) = (\dots,0,A(\seq{v}{t}),0,\dots,0)$.
In particular, $\OPnorm{\seq{\kappa}{t}} = \Norm{\Vc}{\seq{v}{t}}$, and $H_{\Seq{v}}$ has the 1-weighted FMP if and only if $\lim_{t \rightarrow -\infty} \Norm{\Vc}{\seq{v}{t}} = 0$.
\end{example}

In Section \ref{section_main}, we introduced the notation $\pOPnorm{q}{\Seq{\kappa}}$ for $\Seq{\kappa} \in L(\Zc,\Yc)^{\Z_-}$.
We use the same notation for elements in $L(\Zc,\R)^{\Z_-}$, where $L(\Zc,\R)$ is equipped with the operator norm induced by $\norm{\cdot}$ and the absolute value on $\R$.
Given an $L \in \Yc^*$, we denote by $L \circ \Seq{\kappa} = (L \circ \seq{\kappa}{t})_{t \leq 0}$ the element-wise composition.
The proof of the next result uses the well-known duality of sequence spaces.

\begin{proposition}
\label{pcont_to_CR_inf_dim}
	Let $p,q \in [1,\infty]$ be H{\"o}lder conjugates.
Suppose $\{ \Seq{z} \in \mathrm{span}(\ZS) \colon \pnorm{p}{\Seq{z}} \leq 1\} \subseteq \ZS \subseteq c_0(\Zc) \cap \ell^p(\Zc)$.
Then, any $p$-continuous linear functional $H \colon \ZS \rightarrow \Yc$ has a convolution representation satisfying $\pOPnorm{q}{L \circ \Seq{\kappa}} < \infty$ for every $L \in \Yc^*$.
\end{proposition}

\begin{proof}
	By Remark \ref{FMP_comparison} and Lemma \ref{lem_CR}.(i), $H$ has a formal convolution representation with some $\Seq{\kappa} \in L(\Zc,\Yc)^{\Z_-}$.
Since $\ZS$ contains the unit ball of its span, the unique linear extension $\tilde{H}$ of $H$ to the span of $\ZS$ is $p$-continuous.
Denote $\Vc = \ell^p(\ZS)$ if $p < \infty$, respectively $\Vc = c_0(\Zc)$ if $p = \infty$.
Let $L \in \Yc^*$.
Since $\mathrm{span}(\ZS)$ is dense in $(\Vc,\pnorm{p}{\cdot})$ by Assumption \ref{asmpt_sequ_space}.(i), $L \circ \tilde{H}$ is the restriction to $\mathrm{span}(\ZS)$ of an element in the dual space of $(\Vc,\pnorm{p}{\cdot})$.
It is a classical result that $\Seq{\lambda} \mapsto ( \Seq{z} \mapsto \sum_{t \leq 0} \seq{\lambda}{t}(\seq{z}{t}) )$ is an isomorphism $\ell^q(\Zc^*,\OPnorm{\cdot}) \rightarrow (\Vc,\pnorm{p}{\cdot})^*$.
Thus, there exists an element $\Seq{\lambda} \in \ell^q(\Zc^*,\OPnorm{\cdot})$ such that $L \circ \tilde{H}(\Seq{z}) = \sum_{t \leq 0} \seq{\lambda}{t}(\seq{z}{t})$.
This can only be if $L \circ \seq{\kappa}{t} = \seq{\lambda}{t}$ for all $t \in \Z_-$.
In particular, $\pOPnorm{q}{L \circ \Seq{\kappa}} = \pOPnorm{q}{\Seq{\lambda}} < \infty$.
\end{proof}

\subsection*{The finite-dimensional case}

This section develops stronger results for a finite-dimen\-sional codomain $\Yc$.
In this case, we can deduce $\pOPnorm{q}{\Seq{\kappa}} < \infty$ in Proposition \ref{pcont_to_CR_inf_dim} without the dual elements $L \in \Yc^*$.
This is the first key difference between the finite- and the infinite-dimensional case.

\begin{remark}
\label{rem_CR_inf_dim}
	Suppose $\Yc$ is finite-dimensional.
Let $q \in [1,\infty]$.
If $\Seq{\kappa} \in L(\Zc,\Yc)^{\Z_-}$ satisfies $\pOPnorm{q}{L \circ \Seq{\kappa}} < \infty$ for every $L \in \Yc^*$, then $\pOPnorm{q}{\Seq{\kappa}} < \infty$.
This follows from the fact that any two norms on a finite-dimensional vector space are equivalent.
Explicitly, if $d \in \N$ denotes the dimension of $\Yc$, then there exist $L_1,\dots,L_d \in \Yc^*$ and a constant $C > 0$ such that $\OPnorm{\lambda} \leq C \sum_{j=1}^d \OPnorm{L_j \circ \lambda}$ for any $\lambda \in L(\Zc,\Yc)$.
Therefore, $\pOPnorm{q}{\Seq{\kappa}} \leq C d \sum_{j=1}^d \pOPnorm{q}{L_j \circ \Seq{\kappa}} < \infty$ by Jensen's inequality.
\end{remark}

The second key difference between the finite- and the infinite-dimensional case is that a formal convolution representation is automatically a proper one, which we stated in Lemma \ref{lem_CR}.(ii).
This fact is a consequence of the following technical result.

\begin{lemma}
\label{cone_lemma}
	Suppose $\Yc$ is finite-dimensional with dimension $d \in \N$.
Consider a linear functional $H \colon \ZS \rightarrow \Yc$ that has a formal convolution representation.
There exists a constant $c > 0$ independent of $H$ such that for all $\Seq{z} \in \ZS$ there exist disjoint subsets $J_1,\dots,J_{2^d} \subseteq \Z_-$ with
\begin{equation*}
	\sum_{t \leq 0} \Norm{\Yc}{\seq{\kappa}{t}(\seq{z}{t})}
	\leq c \sum_{i=1}^{2^d} \Norm{\Yc}{H\left( \sum_{t \in J_i} \delta^t(\seq{z}{t}) \right)}
	< \infty.
\end{equation*}
\end{lemma}

\begin{proof}
	We make the following claim.
There exists a constant $c > 0$ and there exists a partition of $\Yc$ into cones $C_1,\dots,C_{2^d}$ intersecting pairwise in the origin, that is, $\bigcup_i C_i = \Yc$ and $C_i \cap C_j = \{0\}$ for any $i \neq j$,  such that for any $i \in \{1,\dots,2^d\}$, any $N \in \N$, and any $v_1,\dots,v_N \in C_i$ we have $\sum_{n=1}^N \Norm{\Yc}{v_n} \leq c \Norm{\Yc}{ \sum_{n=1}^N v_i }$.
Now, fix $\Seq{z} \in \ZS$.
Consider $J_i = \{ t \in \Z_- \colon \seq{\kappa}{t}(\seq{z}{t}) \in C_i \backslash \{0\} \}$ and $\Seq{z}^i = \sum_{t \in J_i} \delta^t(\seq{z}{t})$, which belongs to $\ZS$ by Assumption \ref{asmpt_sequ_space}.(ii).
Then,
\begin{equation*}
\begin{split}
	\sum_{t \in J_i} \Norm{\Yc}{ \seq{\kappa}{t}(\seq{z}{t}) }
	&= \liminf_{T \rightarrow -\infty} \sum_{t = T}^0 \Norm{\Yc}{ \seq{\kappa}{t}(\seq{z}{t}^i) }
	\\
	&\leq \liminf_{T \rightarrow -\infty} c \Norm{\Yc}{ \sum_{t = T}^0 \seq{\kappa}{t}(\seq{z}{t}^i) }
	= c \Norm{\Yc}{ H(\Seq{z}^i) }.
\end{split}
\end{equation*}
Thus, $\sum_{t \leq 0} \Norm{\Yc}{ \seq{\kappa}{t}(\seq{z}{t}) } \leq c \sum_{i=1}^{2^d} \Norm{\Yc}{ H(\Seq{z}^i) } < \infty$.
It remains to prove the claim.
It suffices to show it for $\R^d$ with the Euclidean norm since all norms on a finite-dimensional vector space are equivalent.
Let $C_1,\dots,C_{2^d}$ be the orthants.
The boundaries can be assigned in any way so that $\bigcup_i C_i = \R^d$ and $C_i \cap C_j = \{0\}$.
Consider the first orthant $C_1$ containing the vector $v_0 = (1,\dots,1)$.
Let $\pi \colon C_1 \rightarrow \{ r v_0 \colon r \geq 0 \}$ be the orthogonal projection onto the ray of $v_0$.
Then, $\Norm{\R^d}{\pi(v)} \leq \Norm{\R^d}{v}$ for any $v \in C_1$.
Also, $\Norm{\R^d}{v} \leq \sqrt{2} \Norm{\R^d}{\pi(v)}$ because the angle between $v$ and $v_0$ is at most $\pi/4$.
Furthermore, $\Norm{\R^d}{v+w} = \Norm{\R^d}{v} + \Norm{\R^d}{w}$ for any $v,w$ in the ray.
Thus,
\begin{equation*}
\begin{split}
	\sum_{n=1}^N \Norm{\R^d}{v_n}
	&\leq \sqrt{2} \sum_{n=1}^N \Norm{\R^d}{\pi(v_n)}
	= \sqrt{2} \Norm{\R^d}{\sum_{n=1}^N \pi(v_n)}
	\\
	&= \sqrt{2} \Norm{\R^d}{\pi\left( \sum_{n=1}^N v_n \right)}
	\leq \sqrt{2} \Norm{\R^d}{\sum_{n=1}^N v_n}\!\!\!.
\end{split}
\end{equation*}
By symmetry, the same holds for the other orthants.
\end{proof}

\begin{proof}[\PfAlt{Theorem \ref{wFMP_fin_dim}}]
(i)
The first part is Proposition \ref{1wFMP}.
The second part follows from Lemma \ref{lem_CR}, Lemma \ref{CR_to_pcont}, Proposition \ref{pcont_to_CR_inf_dim}, and Remark \ref{rem_CR_inf_dim}.

(ii)
This follows from Remark \ref{FMP_comparison}, Lemma \ref{lem_CR}, Lemma \ref{CR_to_pcont}, Proposition \ref{pcont_to_CR_inf_dim}, and Remark \ref{rem_CR_inf_dim}.

(iii)
It follows from Remark \ref{FMP_comparison} and Lemma \ref{CR_to_pcont} that a convolution representation satisfying $\pOPnorm{1}{\Seq{\kappa}} < \infty$ implies the $\infty$-weighted FMP, which in turn implies the minimal FMP and minimal continuity.
We know from Lemma \ref{lem_CR} that the minimal FMP and minimal continuity imply a convolution representation.
We show that if $H$ has a convolution representation, then it must satisfy $\pOPnorm{1}{\Seq{\kappa}} < \infty$.
Suppose this was not true.
Then, there exists a weighting sequence $\Seq{w}$ such that $\sum_{t \leq 0} \seq{w}{t} \OPnorm{\seq{\kappa}{t}} = \infty$.
For any $t \in \Z_-$, take $z_t \in \Zc$ with norm one and such that $\OPnorm{\seq{\kappa}{t}} \leq 2 \Norm{\Yc}{ \seq{\kappa}{t}(z_t) }$.
Let $\Seq{z} = \sum_{t \leq 0} \delta^t(\seq{w}{t} z_t) \in c_0(\Zc) \cap \B^{\infty}(\Zc)$.
Then, $\sum_{t \leq 0} \seq{w}{t} \OPnorm{\seq{\kappa}{t}} \leq 2 \sum_{t \leq 0} \Norm{\Yc}{\seq{\kappa}{t}(\seq{z}{t})} < \infty$ since the convolution representation is proper -- a contradiction.
Now, assume $\ZS \subseteq c_0(\Zc)$.
We know from Remark \ref{FMP_comparison} that the $\infty$-weighted FMP implies $\infty$-continuity and the latter implies the minimal FMP and minimal continuity, which we already showed to be equivalent to the $\infty$-weighted FMP.
\end{proof}

\begin{proof}[\PfAlt{Proposition \ref{productFMP_fin_dim}}]
	(i)
It is clear that finite memory implies the product FMP.
Conversely, we know from Remark \ref{FMP_comparison} and Lemma \ref{lem_CR}.(i) that the product FMP implies a formal convolution representation with some $\Seq{\kappa} \in L(\Zc,\Yc)^{\Z_-}$.
Suppose for contradiction $H$ has infinite memory.
Then, there exists a monotone sequence $(t_n)_n \subseteq \Z_-$ such that $\seq{\kappa}{t_n} \neq 0$ for all $n \in \N$.
For any $n \in \N$, take $z_n \in \Zc$ such that $\Norm{\Yc}{\seq{\kappa}{t_n}(z_n)} = 1$.
Consider $\Seq{z}^N = \sum_{1 \leq n \leq N} \delta^{t_n}(z_n)$.
By Lemma \ref{cone_lemma}, for any $N \in \N$, there exist $J_1^N,\dots,J_{2^d}^N \subseteq \Z_-$ such that
\begin{equation*}
\begin{split}
	1
	&= \frac{1}{N} \sum_{1 \leq n \leq N} \Norm{\Yc}{\seq{\kappa}{t_n}(z_n)}
	= \frac{1}{N} \sum_{t \leq 0} \Norm{\Yc}{\seq{\kappa}{t}(\seq{z}{t}^N)}
	\\
	&\leq \frac{c}{N} \sum_{i=1}^{2^d} \Norm{\Yc}{H\left( \sum_{t \in J_i^N} \delta^t(\seq{z}{t}^N) \right)}\!\!\!.
\end{split}
\end{equation*}
But $\frac{1}{N} \sum_{t \in J_i^N} \delta^t(\seq{z}{t}^N)$ converges to the zero sequence $\Seq{0}$ in the product topology.
By the product FMP, $H(\frac{1}{N} \sum_{t \in J_i^N} \delta^t(\seq{z}{t}^N)) \rightarrow H(\Seq{0}) = 0$, a contradiction.

(ii)
We need only show that the minimal FMP and minimal continuity imply finite memory.
We know from Lemma \ref{lem_CR} that $H$ has a convolution representation with some $\Seq{\kappa} \in L(\Zc,\Yc)^{\Z_-}$.
Let $J$ be the set of all $t \in \Z_-$ with $\seq{\kappa}{t} \neq 0$.
For any $t \in J$, take $z_t \in \Zc$ such that $\Norm{\Yc}{\seq{\kappa}{t}(z_t)} = 1$.
This time, consider $\Seq{z} = \sum_{t \in J} \delta^t(z_t)$.
Then, since the convolution representation is proper, $\# J = \sum_{t \in J} \Norm{\Yc}{\seq{\kappa}{t}(z_t)} = \sum_{t \leq 0} \Norm{\Yc}{\seq{\kappa}{t}(\seq{z}{t})} < \infty$.
\end{proof}

\begin{remark}
	Suppose $\Yc = \Zc = \R$.
Equip the space of linear functionals $\mathrm{span}(\ZS) \rightarrow \R$ with the weak$^*$ topology, and let $\Vc$ be the largest linear subspace thereof for which the coordinate projections form a Schauder basis.
By definition, a linear functional $\ZS \rightarrow \R$ has a formal convolution representation if and only if it is the restriction to $\ZS$ of an element in $\Vc$.
If $\Hc$ denotes the set of linear functionals $\ZS \rightarrow \R$ that have the minimal FMP and are minimally continuous, then Lemma \ref{lem_CR}.(i) states that $\Hc = \{ H|_{\ZS} \colon H \in \Vc \}$.
Lemma \ref{lem_CR}.(ii) states that the coordinate projections form an unconditional Schauder basis of $\Vc$.
Furthermore, if $\ZS = c_0(\Zc) \cap \ell^p(\R)$ for some $p \in [1,\infty]$, then it is well-known that $\ZS^* \subseteq \Vc$ \cite{LiQueff2017}.
In this case, we deduce from Theorem \ref{wFMP_fin_dim} the reverse inclusion $\Vc \subseteq \ZS^*$ and, hence, equality $\ZS^* = \Vc$.
\end{remark}


\section{Technical details for Section IV}

\begin{proof}[\PfAlt{Proposition \ref{RKHS_l2}}]
	By Lemma \ref{lem_CR}.(i) and Lemma \ref{CR_to_pcont}, the restriction of $H$ to $\B^2(\Zc)$ is 2-continuous and, hence, admits a unique 2-continuous linear extension $\tilde{H} \colon \ell^2(\ZS) \rightarrow \Yc$.
By Proposition \ref{RKHS}, $y \mapsto H_y = \ip{y}{H(\cdot)}_{\Yc}$ is an isomorphism $\Yc_0 \rightarrow \Hb$.
We show that the map $\phi \colon \Yc_0 \rightarrow \ell^2(\Zc)^*$ given by $y \mapsto \ip{y}{\tilde{H}(\cdot)}_{\Yc}$ is injective.
Suppose $\ip{y}{\tilde{H}(\cdot)}_{\Yc} = 0$.
Then, $H_y$ is constantly zero on $\B^2(\Zc)$.
The minimal FMP and minimal continuity of $H$ transfer to $H_y$.
By Lemma \ref{lem_CR}.(i), $H_y$ has a formal convolution representation $\Seq{\kappa}$.
This $\Seq{\kappa}$ must be trivial since it is uniquely determined by the restriction of $H_y$ to $\B^2(\Zc)$.
Thus, $H_y$ is constantly zero on all of $\ZS$, which can only be if $y=0$.
This shows injectivity of $\phi$.
Note that the orthogonal complement of $\tilde{H}(\ell^2(\Zc))$ in $\Yc_0$ belongs to the kernel of $\phi$ and, hence, is trivial.
Having a trivial orthogonal complement is equivalent to density of $\tilde{H}(\ell^2(\Zc))$ in $\Yc_0$.

Now, assume $\tilde{H}$ is orthogonal.
Then, $\tilde{H}(\ell^2(\Zc))$ is complete.
As it is also dense in $\Yc_0$, we must have $\tilde{H}(\ell^2(\Zc)) = \Yc_0$.
Given any $y_1,y_2 \in \Yc_0$, take $\Seq{z}^{y_1},\Seq{z}^{y_2} \in \ell^2(\Zc)$ with $\tilde{H}(\Seq{z}^{y_i}) = y_i$.
By orthogonality of $\tilde{H}$, $\phi(y_i) = \ip{\Seq{z}^{y_i}}{\cdot}_{\ell^2(\Zc)}$ and, hence,
\begin{equation*}
\begin{split}
	\ip{\phi(y_1)}{\phi(y_2)}_{\ell^2(\Zc)^*}
	&= \ip{\Seq{z}^{y_1}}{\Seq{z}^{y_2}}_{\ell^2(\Zc)}
	\\
	&= \ip{\tilde{H}(\Seq{z}^{y_1})}{\tilde{H}(\Seq{z}^{y_2})}_{\Yc}
	= \ip{y_1}{y_2}_{\Yc}.
\end{split}
\end{equation*}
This shows that $\phi$ is an isometry.
To see that $\phi$ is surjective, let $f \in \ell^2(\Zc)^*$.
By the Riesz representation theorem, there exists some $\Seq{z}^0 \in \ell^2(\Zc)$ such that $f = \ip{\Seq{z}^0}{\cdot}_{\ell^2(\Zc)}$.
Then, $f = \phi(\tilde{H}(\Seq{z}^0))$.
\end{proof}

\begin{corollary}
	Suppose $\ZS$ is a normed vector space that contains $\ell^2(\Zc)$ as a non-dense subset and such that $\delta^t \colon \Bc \rightarrow \ZS$ is continuous for all $t \in \Z_-$.
Then, there exists an element $\Seq{z} \in \ZS$ such that the truncated sequence $\sum_{t=T}^0 \delta^t(\seq{z}{t})$ does not converge to $\Seq{z}$ as $T \rightarrow -\infty$.
\end{corollary}

\begin{proof}
	The Hahn-Banach theorem guarantees the existence of a non-trivial, continuous, linear functional $H \colon \ZS \rightarrow \R$ that vanishes on $\ell^2(\Zc)$.
In particular, $H(\ell^2(\Zc)) = 0$ is not dense in $H(\ZS) = \R$.
By Proposition \ref{RKHS_l2}, $H$ cannot have the minimal FMP.
That is, there exists an element $\Seq{z} \in \ZS$ such that $H(\sum_{t=T}^0 \delta^t(\seq{z}{t}))$ does not converge to $H(\Seq{z})$.
Since $H$ is continuous, it follows that $\sum_{t=T}^0 \delta^t(\seq{z}{t})$ cannot converge to $\Seq{z}$.
\end{proof}


\section{Functionals and filters}
\label{appendix_funfil}

We recall the equivalence between functionals and time-invariant filters \cite{BoydChua1985, RC9}.
In this section, $\Yc$ and $\Zc$ are simply sets.
Consider $\ZS^{\pm} \subseteq \Zc^{\Z}$ and $\ZS^- \subseteq \Zc^{\Z_-}$.
Let $\CTI$ and $\TI$ denote the set of all causal and time-invariant filters $\ZS^{\pm} \rightarrow \Yc^{\Z}$, respectively the set of all time-invariant filters $\ZS^- \rightarrow \Yc^{\Z_-}$ (which are automatically causal).
Let $\Func$ denote the set of all functionals $\ZS^- \rightarrow \Yc$.
For clarity, we distinguish between the shift $T \colon \Zc^{\Z} \rightarrow \Zc^{\Z}$ on bi-infinite sequences and the shift $T_- \colon \Zc^{\Z_-} \rightarrow \Zc^{\Z_-}$ on semi-infinite sequences.
Let $\tau_{\Yc} \colon \Yc^{\Z} \rightarrow \Yc^{\Z_-}$ and $\tau_{\Zc} \colon \Zc^{\Z} \rightarrow \Zc^{\Z_-}$ be the truncations, and let $p^t(\Seq{y}) = \seq{y}{t}$ be the projection onto the time entry $t$.

\begin{lemma}
	Suppose $T(\ZS^{\pm}) = \ZS^{\pm}$ and $\tau_{\Zc}(\ZS^{\pm}) = \ZS^-$.
Let $\iota \colon \ZS^- \rightarrow \ZS^{\pm}$ be any right-inverse of $\tau_{\Zc}$.
Then, there is a commutative diagram of bijections
\begin{equation*}
\begin{tikzcd}[row sep = 3em, column sep = 3em]
	\CTI \arrow{dr}[swap, sloped]{V \mapsto H_V} \arrow{rr}{V \mapsto U_V} \pgfmatrixnextcell \pgfmatrixnextcell \arrow{ll}{V_U \mapsot U} \TI \arrow{dl}[sloped]{H_U \mapsot U}
	\\
	\pgfmatrixnextcell \arrow{ul}[sloped]{V_H \mapsot H} \Func \arrow{ur}[swap, sloped]{H \mapsto U_H}
\end{tikzcd}
\end{equation*}
in which the maps are given (component-wise) by
\begin{equation*}
\begin{split}
	U_V
	&= \tau_{\Yc} \circ V \circ \iota,
	\\
	H_V
	&= p^0 \circ V \circ \iota,
	\\
	H_U
	&= p^0 \circ U,
\end{split}
\qquad
\begin{split}
	p^t \circ V_U
	&= p^0 \circ U \circ \tau_{\Zc} \circ T^{-t},
	\\
	p^t \circ V_H
	&= H \circ \tau_{\Zc} \circ T^{-t},
	\\
	p^t \circ U_H
	&= H \circ T_-^{-t}.
\end{split}
\end{equation*}
These bijections are independent of the choice of $\iota$.
\end{lemma}

\begin{proof}
	We have to verify that all maps are well-defined.
First, observe that $T_-^{-t}(\ZS^-) = \ZS^-$ for any $t \in \Z_-$, which follows from the equality $\tau_{\Zc} \circ T^{-t} = T_-^{-t} \circ \tau_{\Zc}$ together with the assumptions that $\ZS^- = \tau_{\Zc}(\ZS^{\pm})$ and $T(\ZS^{\pm}) = \ZS^{\pm}$.
Since $\Func$ contains all maps $\ZS^- \rightarrow \Yc$, it is clear that $V \mapsto H_V$ and $U \mapsto H_U$ are well-defined.
Note that a filter $U \colon \ZS^- \rightarrow \Yc^{\Z_-}$ is time-invariant if and only if $p^t \circ U \circ T_-^{-s} = p^{t+s} \circ U$ for any $t,s \in \Z_-$.
From this, it is easy to see that $H \mapsto U_H$ maps into $\TI$.
Analogously, a filter $V \colon \ZS^{\pm} \rightarrow \Yc^{\Z}$ is time-invariant if and only if $p^t \circ V \circ T^{-s} = p^{t+s} \circ V$ for any $t,s \in \Z$, from which it follows that $V_U$ and $V_H$ are time-invariant for any filter $U$ and any functional $H$.
Next, observe that a filter $V \colon \ZS^{\pm} \rightarrow \Yc^{\Z}$ is causal if and only if $p^t \circ V(\Seq{z}^1) = p^t \circ V(\Seq{z}^2)$ for any $t \in \Z$ and $\Seq{z}^1,\Seq{z}^2 \in \ZS^{\pm}$ that satisfy $\tau_{\Zc} \circ T^{-t}(\Seq{z}^1) = \tau_{\Zc} \circ T^{-t}(\Seq{z}^2)$.
From this, we find that $V_U$ and $V_H$ are also causal and, hence, that $U \mapsto V_U$ and $H \mapsto V_H$ map into $\CTI$.
It remains to show that $U_V$ is time-invariant for any $V \in \CTI$.
The causality and time-invariance of $V$ imply for any $s,t \in \Z_-$ that
\begin{equation*}
	p^t \circ V \circ \iota \circ T_-^{-s}
	= p^t \circ V \circ T^{-s} \circ \iota
	= p^{t+s} \circ V \circ \iota.
\end{equation*}
Clearly, $p^t \circ \tau_{\Yc} = p^t$ for any $t \in \Z_-$.
Thus,
\begin{equation*}
\begin{split}
	p^t \circ U_V \circ T_-^{-s}
	&= p^t \circ \tau_{\Yc} \circ V \circ \iota \circ T_-^{-s}
	\\
	&= p^t \circ V \circ \iota \circ T_-^{-s}
	= p^{t+s} \circ V \circ \iota
	\\
	&= p^{t+s} \circ \tau_{\Yc} \circ V \circ \iota
	= p^{t+s} \circ U_V,
\end{split}
\end{equation*}
which shows that $V \mapsto V_U$ maps into $\TI$.
We established that all maps are well-defined.
To show that the maps in the diagram are bijective, we verify that the formulas for maps in opposite directions in the diagram are indeed left- and right-inverses of each other.
The map $\iota$ may not be a left-inverse of $\tau_{\Zc}$, but the causality of an element $V \in \CTI$ implies that $p^0 \circ V \circ \iota \circ \tau_{\Zc} = p^0 \circ V$.
Thus, for any $t \in \Z$ and $s \in \Z_-$,
\begin{equation*}
\begin{split}
	p^t \circ V_{U_V}
	&= p^0 \circ U_V \circ \tau_{\Zc} \circ T^{-t}
	= p^0 \circ \tau_{\Yc} \circ V \circ \iota \circ \tau_{\Zc} \circ T^{-t}
	\\
	&= p^0 \circ V \circ T^{-t}
	= p^t \circ V,
	\\
	p^s \circ U_{V_U}
	&= p^s \circ \tau_{\Yc} \circ V_U \circ \iota
	= p^s \circ V_U \circ \iota
	\\
	&= p^0 \circ U \circ \tau_{\Zc} \circ T^{-s} \circ \iota
	= p^0 \circ U \circ T_-^{-s}
	= p^s \circ U.
\end{split}
\end{equation*}
Similarly, for any $t \in \Z$,
\begin{equation*}
\begin{split}
	p^t \circ V_{H_V}
	&= H_V \circ \tau_{\Zc} \circ T^{-t}
	= p^0 \circ V \circ \iota \circ \tau_{\Zc} \circ T^{-t}
	\\
	&= p^0 \circ V \circ T^{-t}
	= p^t \circ V,
	\\
	H_{V_H}
	&= p^0 \circ V_H \circ \iota
	= H \circ \tau_{\Zc} \circ T^0 \circ \iota
	= H.
\end{split}
\end{equation*}
Last but not least, for any $t \in \Z_-$,
\begin{equation*}
\begin{split}
	p^t \circ U_{H_U}
	&= H_U \circ T_-^{-t}
	= p^0 \circ U \circ T_-^{-t}
	= p^t \circ U,
	\\
	H_{U_H}
	&= p^0 \circ U_H
	= H \circ T_-^0
	= H.
\end{split}
\end{equation*}
This shows that all maps are bijections with inverses as claimed.
Since the maps $V_U$, $V_H$, and $U_H$ do not depend on the choice of $\iota$, neither do their inverses.
Commutativity of the diagram amounts to the chain of equalities
\begin{equation*}
	H_{U_V}
	= p^0 \circ U_V
	= p^0 \circ \tau_{\Yc} \circ V \circ \iota
	= p^0 \circ V \circ \iota
	= H_V.
\end{equation*}
\end{proof}

\begin{remark}
	Let $\Fc(\ZS^{\pm})$ and $\Fc(\ZS^-)$ be the set of all filters $\ZS^{\pm} \rightarrow \Yc^\Z$, respectively $\ZS^- \rightarrow \Yc^{\Z_-}$.
The filter $V_U$ as defined in the previous lemma is causal and time-invariant for any $U \in \Fc(\ZS^-)$ even if $U$ is not time-invariant, and the map $\TI \rightarrow \CTI$, $U \mapsto V_U$ extends to a map $\Fc(\ZS^-) \rightarrow \CTI$.
It is clear that $\CTI \rightarrow \Func$, $V \mapsto H_V$ also extends to a map $\Fc(\ZS^{\pm}) \rightarrow \Func$, and likewise for $U \mapsto H_U$.
Thus, we actually have a larger commutative diagram as follows.
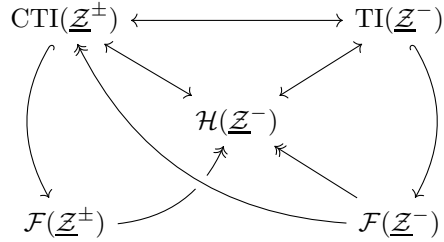
\begin{equation*}
\begin{tikzcd}[row sep = 2em, column sep = 2em]
	\arrow[hook, bend right]{dd} \CTI \arrow[leftrightarrow]{rr} \pgfmatrixnextcell \pgfmatrixnextcell \TI \arrow[hook', bend left]{dd}
	\\
	\pgfmatrixnextcell \arrow[leftrightarrow]{ul} \Func \arrow[leftrightarrow]{ur}
	\\
	\Fc(\ZS^{\pm}) \arrow[two heads, bend right]{ur} \pgfmatrixnextcell \pgfmatrixnextcell \arrow[two heads, bend left, crossing over]{uull} \arrow[two heads]{ul} \Fc(\ZS^-)
\end{tikzcd}
\end{equation*}
The concatenations $\Fc(\ZS^{\pm}) \rightarrow \Fc(\ZS^{\pm})$, $V \mapsto V_{H_V}$ and $\Fc(\ZS^-) \rightarrow \Fc(\ZS^-)$, $U \mapsto U_{H_U}$ are idempotent maps and define natural projections $\Fc(\ZS^{\pm}) \rightarrow \CTI$ and $\Fc(\ZS^-) \rightarrow \TI$ from the set of all filters onto the set of (causal) time-invariant ones.
\end{remark}

\begin{remark}
	It should be noted that, in general, it is \textit{not} true that a filter is continuous if and only if its associated functional is continuous.
This hinges on the choice of topology on the codomain of the filter.
If the codomain is endowed with the product topology, then a filter is indeed continuous if and only if its associated functional is continuous.
For other topologies on the codomain, this becomes a subtle issue \cite{RC7,RC9}.
\end{remark}


\section{Continuous-time linear functionals}
\label{appendix_cont}

In this appendix, we make precise a claim made in \cref{Conclusion}.
Let $\Ib \subseteq \R$ be an interval with non-empty interior, and let $C_c^{\infty}(\Ib,\R^d)$ be the set of smooth and compactly supported functions $\Ib \rightarrow \R^d$.
Recall that the topology of point-wise convergence on $C_c^{\infty}(\Ib,\R^d)$ is the subspace topology induced by the product topology on $(\R^d)^{\Ib}$.

\begin{proposition}
\label{cont_time_conv_repr}
	Let $\kappa \colon \Ib \rightarrow \R^d$ be locally integrable, and consider the linear functional $H \colon C_c^{\infty}(\Ib,\R^d) \rightarrow \R$,
\begin{equation*}
	H(f)
	= \int_{\Ib} \kappa(t)^T f(t) \, dt.
\end{equation*}
Suppose $H$ is continuous with respect to the topology of point-wise convergence.
Then, $\kappa$ vanishes Lebesgue-almost everywhere and, hence, $H$ is constantly zero.
\end{proposition}

\begin{proof}
	Suppose for contradiction that $\kappa$ is non-zero on a set of positive Lebesgue measure.
Consider the family of standard symmetric mollifiers $(\phi_{\epsilon})_{\epsilon \in (0,1)} \subseteq C_c^{\infty}(\R,\R)$, and recall that the support of each $\phi_{\epsilon}$ is contained in $[-\epsilon,+\epsilon]$.
Let $\norm{\cdot}$ be the Euclidean norm.
Since the convolution $\phi_{\epsilon} * \kappa$ converges point-wise Lebesgue-almost everywhere to $\kappa$ as $\epsilon \rightarrow 0$, there exists a point $t_0$ in the interior of $\Ib$ and constants $\epsilon_0,\eta > 0$ such that $\norm{\phi_{\epsilon} * \kappa(t_0)} > \eta$ for all $\epsilon \in (0,\epsilon_0)$.
Since $\phi_{\epsilon} * \kappa$ is continuous, for all $\epsilon \in (0,\epsilon_0)$ there exists some $\delta_{\epsilon} \in (0,\epsilon)$ such that $\norm{\phi_{\epsilon} * \kappa(t)} > \eta$ for all $t \in (t_0 - 3\delta_{\epsilon},t_0)$.
Let $\tau_{\epsilon} \in C_c^{\infty}(\R,[0,1])$ be a bump function that is 1 in $(t_0 - 2\delta_{\epsilon},t_0 - \delta_{\epsilon})$ and is zero outside of $(t_0 - 3\delta_{\epsilon},t_0)$.
Define $h_{\epsilon} \in C_c^{\infty}(\R,\R^d)$ by $h_{\epsilon}(t) = (\delta_{\epsilon})^{-1} \tau_{\epsilon}(t) (\phi_{\epsilon} * \kappa)(t)$, and set $g_{\epsilon} = \phi_{\epsilon} * h_{\epsilon} \in C_c^{\infty}(\R,\R^d)$.
The support of $g_{\epsilon}$ is contained in $[t_0-4{\epsilon},t_0+{\epsilon}]$.
In particular, since $t_0$ is an interior point of $\Ib$, we have $g_{\epsilon} \in C_c(\Ib,\R^d)$ for sufficiently small $\epsilon$.
Let us compute $H(g_{\epsilon})$.
By Fubini's theorem and by symmetry of the mollifiers, that is, $\phi_{\epsilon}(t-s) = \phi_{\epsilon}(s-t)$, we have
\begin{equation*}
\begin{split}
	H(g_{\epsilon})
	&= \int_{\Ib} \kappa(t)^T (\phi_{\epsilon} * h_{\epsilon})(t) \, dt
	\\
	&= \int_{\Ib} \int_{\R} \kappa(t)^T \phi_{\epsilon}(s-t) h_{\epsilon}(s) \, ds \, dt
	\\
	&= \int_{\R} (\phi_{\epsilon} * \kappa)(s)^T h_{\epsilon}(s) \, ds.
\end{split}
\end{equation*}
Unravelling the definitions, we find
\begin{equation*}
\begin{split}
	H(g_{\epsilon})
	&= \int_{\R} (\phi_{\epsilon} * \kappa)(s)^T h_{\epsilon}(s) \, ds
	\\
	&= \frac{1}{\delta_{\epsilon}} \int_{t_0-3\delta_{\epsilon}}^{t_0} \tau_{\epsilon}(s) \norm{ (\phi_{\epsilon} * \kappa)(s)}^2 ds
	\\
	&\geq \frac{\eta^2}{\delta_{\epsilon}} \int_{t_0-3\delta_{\epsilon}}^{t_0} \tau_{\epsilon}(s) \, ds
	\geq \eta^2.
\end{split}
\end{equation*}
Next, for any $\epsilon \in (0,\epsilon_0)$, take $r_{\epsilon} \in (0,\epsilon)$ so small that
\begin{equation*}
	\int_{t_0-2r_{\epsilon}}^{t_0+2r_{\epsilon}} \abs{ \kappa(t)^T g_{\epsilon}(t) } dt
	\leq \epsilon.
\end{equation*}
Let $\hat{\tau}_{\epsilon} \in C_c^{\infty}(\R,[0,1])$ be a bump function that is 1 in $(t_0-r_{\epsilon},t_0+r_{\epsilon})$ and is zero outside of $(t_0-2r_{\epsilon},t_0+2r_{\epsilon})$.
Then,
\begin{equation*}
\begin{split}
	\abs{ H(\hat{\tau}_{\epsilon} g_{\epsilon}) }
	&\leq \int_{\Ib} \hat{\tau}_{\epsilon}(t) \abs{ \kappa(t)^T g_{\epsilon}(t)  } dt
	\\
	&\leq \int_{t_0-2r_{\epsilon}}^{t_0+2r_{\epsilon}} \abs{ \kappa(t)^T g_{\epsilon}(t) } dt
	\leq \epsilon.
\end{split}
\end{equation*}
Finally, set $f_{\epsilon} = (1 - \hat{\tau}_{\epsilon}) g_{\epsilon} \in C_c^{\infty}(\R,\R^d)$.
Clearly, $f_{\epsilon}(t_0) = 0$.
That the support of $f_{\epsilon}$ is contained in $[t_0-4{\epsilon},t_0+{\epsilon}]$ implies that $(f_{\epsilon}(t))_{\epsilon}$ is eventually zero for any $t \neq t_0$.
Thus, $H(f_{\epsilon}) \rightarrow 0$ by continuity of $H$ with respect to the topology of point-wise convergence.
This contradicts $H(f_{\epsilon}) = H(g_{\epsilon}) - H(\hat{\tau}_{\epsilon} g_{\epsilon}) \geq \eta^2 - \epsilon$.
\end{proof}

Let $C_0(\Ib,\R^d)$ denote the set of continuous functions vanishing at infinity.
Consider the set $\Rc$ of linear functionals $H \colon C_0(\Ib,\R^d) \rightarrow \R$ that satisfy $H(f_n) \rightarrow 0$ for any sequence $(f_n)_n \subseteq C_0(\Ib,\R^d)$ that converges to zero point-wise Lebesgue-almost everywhere and is uniformly bounded in the sense that $\sup_n \sup_{t \in \Ib} \Norm{\R^d}{f_n(t)} < \infty$.
Note that any element in $\Rc$ is continuous with respect to the supremums-norm.
Thus, $\Rc$ is a subset of the dual space $C_0(\Ib,\R^d)^*$.
It can be shown that $\Rc$ is exactly the set of functionals one obtains by embedding $L^1(\Ib,\R^d)$ in the dual $C_0(\Ib,\R^d)^*$ \cite{LiHanELi2020, LiHanELi2022}.
Proposition \ref{cont_time_conv_repr} implies that the condition $\sup_n \sup_{t \in \Ib} \Norm{\R^d}{f_n(t)} < \infty$ is crucial in the property above.
Indeed, suppose $H \colon C_0(\Ib,\R^d) \rightarrow \R$ is a linear functional that satisfies $H(f_n) \rightarrow 0$ for any sequence $(f_n)_n \subseteq C_0(\Ib,\R^d)$ that converges to zero point-wise Lebesgue-almost everywhere but is permitted to be unbounded.
Then, $H \in \Rc \cong L^1(\Ib,\R^d)$ admits an integral representation as required in Proposition \ref{cont_time_conv_repr}, and it is continuous with respect to the topology of point-wise convergence.
Hence, $H$ is constantly zero.
We remark that the property specifying elements in $\Rc$ cannot be defined as continuity of the functional with respect to a topology.

\bib{acm}{bibfile_FR}

\end{document}